\RequirePackage{ifpdf} 
\ifpdf
\documentclass[11pt,pdftex]{article}
\else
\documentclass[11pt,dvips]{article}
\fi

\ifpdf
\fi
\usepackage[bookmarks, 
colorlinks=false, backref,plainpages=false]{hyperref}
\usepackage[english]{babel}
\usepackage{graphics}
\usepackage{epsfig}
\usepackage{amsmath,amssymb,amsthm}

\renewcommand{\theequation}{\thesection.\arabic{equation}}

\newtheorem{thm}{Theorem}[section]

\newtheorem{lem}[thm]{Lemma}

\newtheorem{defn}[thm]{Definition}
\newtheorem{rem}[thm]{Remark}

\begin{document}
\newcommand{\BX}{{\bf X}}
\newcommand{\cv}{{\cal V}}
\newcommand{\cW}{{\cal W}}
\newcommand{\co}{{\cal O}}

\renewcommand{\theequation}{\thesection.\arabic{equation}}
\def\@eqnnum{{\reset@font\rm (\theequation)}}

\newtheorem{theorem}{Theorem}[section]
\newtheorem{corollary}{Corollary}[section]
\newtheorem{lemma}{Lemma}[section]
\newtheorem{proposition}{Proposition}[section]
\newtheorem{conjecture}{Conjecture}[section]
\newtheorem{remark}{Remark}[section]
\newtheorem{definition}{Definition}[section]
\newtheorem{problem}{Problem}[section]

\def\abstract{
\advance \rightskip by 10mm
\advance \leftskip by 10mm
\vspace{-0.8em}
\noindent
\small{\bf Abstract.}
}
\def\endabstract{\par\normalsize\rm}

\def\Xint#1{\mathchoice
{\XXint\displaystyle\textstyle{#1}}%
{\XXint\textstyle\scriptstyle{#1}}%
{\XXint\scriptstyle\scriptscriptstyle{#1}}%
{\XXint\scriptscriptstyle\scriptscriptstyle{#1}}%
\!\int}
\def\XXint#1#2#3{{\setbox0=\hbox{$#1{#2#3}{\int}$}
\vcenter{\hbox{$#2#3$}}\kern-.5\wd0}}
\def\ddashint{\Xint=}
\def\dashint{\Xint-}

\def\a{\alpha}
\def\b{\beta}
\def\d{\delta}\def\D{\Delta}
\def\e{\epsilon}
\def\g{\gamma}\def\G{\Gamma}
\def\k{\kappa}
\def\lam{\lambda}\def\Lam{\Lambda}
\renewcommand\o{\omega}\renewcommand\O{\Omega}
\def\s{\sigma}\def\S{\Sigma}
\renewcommand\t{\theta}\def\vt{\vartheta}
\newcommand{\vphi}{\varphi}
\def\z{\zeta}

\newcommand{\tsigma}{\tilde{\s}}
\newcommand{\tbsigma}{\tilde{\bsigma}}
\def\te{\tilde{\e}}
\def\tu{\tilde{u}}

\newcommand{\bchi}{\mbox{\boldmath$\chi$}}
\newcommand{\bdelta}{\mbox{\boldmath$\delta$}}
\newcommand{\bepsilon}{\mbox{\boldmath$\epsilon$}}
\newcommand{\bfeta}{\mbox{\boldmath$\eta$}}
\newcommand{\bgamma}{\mbox{\boldmath$\gamma$}}
\newcommand{\bomega}{\mbox{\boldmath$\omega$}}
\newcommand{\bvphi}{\mbox{\boldmath$\varphi$}}
\newcommand{\bphi}{\mbox{\boldmath$\phi$}}
\newcommand{\bPhi}{\mbox{\boldmath$\Phi$}}
\newcommand{\bpsi}{\mbox{\boldmath$\psi$}}
\newcommand{\bPsi}{\mbox{\boldmath$\Psi$}}
\newcommand{\bsigma}{\mbox{\boldmath$\sigma$}}
\newcommand{\btau}{\mbox{\boldmath$\tau$}}
\newcommand{\bxi}{\mbox{\boldmath$\xi$}}
\newcommand{\brho}{\mbox{\boldmath$\rho$}}

\def\bk{\boldsymbol{\kappa}}
\def\bmu{\boldsymbol\mu}
\def\bxi{\boldsymbol{\xi}}
\def\bz{\boldsymbol{\zeta}}

\def\ba{{\bf a}}
\def\bb{{\bf b}}
\def\bc{{\bf c}}
\def\be{{\bf e}}
\def\bff{{\bf f}}
\def\bg{{\bf g}}
\def\bn{{\bf n}}
\def\bp{{\bf p}}
\def\bq{{\bf q}}
\def\bs{{\bf s}}
\def\bt{{\bf t}}
\def\bu{{\bf u}}
\def\bv{{\bf v}}
\def\bw{{\bf w}}
\def\bx{{\bf x}}
\def\by{{\bf y}}
\def\bzz{{\bf z}}

\def\bD{{\bf D}}
\def\bE{{\bf E}}
\def\bF{{\bf F}}
\def\bH{{\bf H}}
\def\bJ{{\bf J}}
\def\bV{{\bf V}}
\def\bU{{\bf U}}
\def\bW{{\bf W}}
\def\bX{{\bf X}}
\def\bY{{\bf Y}}

\def\cA{{\cal A}}
\def\cC{{\cal C}}
\def\cD{{\cal D}}
\def\cE{{\cal E}}
\def\cF{{\cal F}}
\def\cG{{\cal G}}
\def\cI{{\cal I}}
\def\cJ{{\cal J}}
\def\cK{{\cal K}}
\def\cL{{\cal L}}
\def\cO{{\cal O}}
\def\cP{{\cal P}}
\def\cQ{{\cal Q}}
\def\cR{{\cal R}}
\def\cS{{\cal \Sigma}}
\def\cT{{\cal T}}
\def\cU{{\cal U}}
\def\cV{{\cal V}}

\def\scT{{_\cT}}
\def\sD{{_D}}
\def\sE{{_E}}
\def\sF{{_F}}
\def\sFz{{_{F_z}}}
\def\sK{{_K}}
\def\sI{{_I}}
\def\sb{{_b}}
\def\sN{{_N}}

\def\curl{{{\bf curl} \ }}
\def\rot{{\mbox{rot}\ }}
\def\BPI{{\bf \Pi}}

\def\cth{\cT_h}
\def\ctH{\cT_H}

\def\tJ{\tilde{\J}}

\def\hK{\widehat{K}}
\def\hx{\widehat{x}}
\def\hy{\widehat{y}}
\def\bhv{\widehat{\bv}}

\def\l{\ell}
\def\bl{\boldsymbol{\ell}}
\def\col{\colon}
\def\f12{\frac12}
\def\dfrac{\displaystyle\frac}
\def\dint{\displaystyle\int}
\def\nab{\nabla}
\def\p{\partial}
\def\sm{\setminus}
\def\dsum{\displaystyle\sum}
\newcommand{\pp}[2]{\frac{\partial {#1}}{\partial {#2}}}
\def\bzero{{\bf 0}}

\def\divv{\nab\cdot}
\def\divx{\nab_x\cdot}
\def\divtx{\nab_{t,x}\cdot}
\def\nabx{\nab_x}

\newcommand{\grad}{\nabla}
\newcommand{\curlt}{{\nabla \times}}
\newcommand{\gperp}{\nabla^{\perp}}
\newcommand{\gradt}{\nabla\cdot}

\def\forallqq{\quad\forall\,}
\def\aph{A^{1/2}}
\def\amh{A^{-1/2}}

\def\osc{{\rm osc \, }}

\def\Im{{\rm Im}}
\newcommand{\tr}{{\rm tr}}
\def\divvr{{\rm div}}
\def\curllr{{\rm curl}}
\def\curll{{\rm curl}}
\def\curl{{\bf curl}}
\newcommand{\bgrad}{{\bf grad}}
\newcommand\diam{\mathrm{diam\,}}
\renewcommand\Im{\mathrm{Im\,}}
\def\Span{\mbox{Span}}
\def\supp{\mbox{supp\,}}
\newcommand{\trace}{{\rm trace}}

\newcommand{\tri}{|\!|\!|}
\newcommand{\ljump}{\lbrack\!\lbrack}
\newcommand{\rjump}{\rbrack\!\rbrack}
\newcommand{\bdm}{\begin{displaymath}}
\newcommand{\edm}{\end{displaymath}}
\newcommand{\beq}{\begin{equation}}
\newcommand{\eeq}{\end{equation}}
\newcommand{\beqa}{\begin{eqnarray}}
\newcommand{\eeqa}{\end{eqnarray}}
\newcommand{\beqas}{\begin{eqnarray*}}
\newcommand{\eeqas}{\end{eqnarray*}}
\newcommand{\ul}{\underline}
\newcommand{\wh}{\widehat}
\newcommand{\la}{\langle}
\newcommand{\ra}{\rangle}

\newcommand{\Lt}{L^2(\Omega)}
\newcommand{\Lts}{L^2(\Omega)^2}
\newcommand{\Ltc}{L^2(\Omega)^3}
\newcommand{\Ho}{H^1(\Omega)}
\newcommand{\Hoh}{H^1(\wh{\Omega})}
\newcommand{\Hoi}{H^1(\Omega_i)}
\newcommand{\Hos}{H^1(\Omega)^2}
\newcommand{\Hoc}{H^1(\Omega)^3}
\newcommand{\Hoch}{H^1(\wh{\Omega})^3}
\newcommand{\Hoci}{H^1(\Omega_i)^3}
\newcommand{\Hoz}{H^1_0(\Omega)}
\newcommand{\Ht}{H^2(\Omega)}
\newcommand{\Hti}{H^2(\Omega_i)}
\newcommand{\Hts}{H^2(\Omega)^2}
\newcommand{\Htc}{H^2(\Omega)^3}
\newcommand{\Htz}{H^0(\Omega)}
\newcommand{\Hh}{H^{1/2}(\Gamma)}
\newcommand{\Hhi}{H^{1/2}(\Gamma_i)}
\newcommand{\Hmh}{H^{-1/2}(\Gamma)}
\newcommand{\Hdiv}{H(\divvr;\,\Omega)}
\newcommand{\Hdivh}{H(\divv;\,\wh \Omega)}
\newcommand{\hcurl}{H(\curl\,A;\,\Omega)}
\newcommand{\Hcurl}{H(\curll\,A;\,\Omega)}
\newcommand{\Hcrl}{H(\curll\,;\,\Omega)}
\newcommand{\hcrl}{H(\curl\,;\,\Omega)}
\newcommand{\Hcrlh}{H(\curll\,;\,\wh\Omega)}
\newcommand{\hcrlh}{H(\curl\,;\,\wh\Omega)}
\newcommand{\Wdiv}{\BW_0(\mbox{\divv}\,;\,\Omega)}
\newcommand{\Wcurl}{\BW_0(\mbox{\curl}\,A;\,\Omega)}
\newcommand{\WcrossV}{\BW \times V}

\def\calS{{\cal S}}
\def\calT{{\cal T}}
\def\cB{{\cal B}}
\def\cH{{\cal H}}
\def\ba{{\mathbf{a}}}
\def\cM{{\cal M}}
\def\cN{{\cal N}}

\def\bE{{\bf E}}
\def\bS{{\bf S}}
\def\br{{\bf r}}
\def\bW{{\bf W}}
\def\bLambda{{\bf \Lambda}}

\newcommand{\lJump}{[\![}
\newcommand{\rJump}{]\!]}
\newcommand{\jump}[1]{[\![ #1]\!]}

\newcommand{\sd}{\bsigma^{\Delta}}
\newcommand{\st}{\tilde{\bsigma}}
\newcommand{\sh}{\hat{\bsigma}}
\newcommand{\rd}{\brho^{\Delta}}

\newcommand{\WH}{W\!H}
\newcommand{\NE}{N\!E}

\newcommand{\ND}{N\!D}
\newcommand{\BDM}{B\!D\!M}

\newcommand{\sT}{{_T}}
\newcommand{\sRT}{{_{RT}}}
\newcommand{\sBDM}{{_{BDM}}}
\newcommand{\sWH}{{_{WH}}}
\newcommand{\sND}{{_{ND}}}
\newcommand{\sV}{_\cV}

\newcommand{\dd}{\underline{{\mathbf d}}}
\newcommand{\C}{\rm I\kern-.5emC}
\newcommand{\R}{\rm I\kern-.19emR}
\newcommand{\W}{{\mathbf W}}
\def\3bar{{|\hspace{-.02in}|\hspace{-.02in}|}}
\newcommand{\A}{{\mathcal A}}

\title {Recovery-Based Error Estimators for Diffusion Problems:\\
Explicit Formulas}
\author{Zhiqiang Cai\thanks{
Department of Mathematics, Purdue University, 150 N. University
Street, West Lafayette, IN 47907-2067, caiz@purdue.edu.
This work was supported in part by the National Science Foundation
under grant DMS-1217081.}
\and Shun Zhang\thanks{Department of Mathematics,
City University of Hong Kong, Kowloon, Hong Kong,
shun.zhang@cityu.edu.hk.}}
 \date{\today}
 \maketitle

\begin{abstract}
We introduced and analyzed robust
recovery-based a posteriori error estimators for various lower order
finite element approximations to interface problems in
\cite{CaZh:09, CaZh:10a}, where the recoveries of the flux
and/or gradient are implicit (i.e., requiring solutions of global problems with mass
matrices). In this paper, we develop fully explicit
recovery-based error estimators for lower order 
conforming, mixed, and nonconforming finite element 
approximations to 
diffusion problems with full coefficient tensor.
When the diffusion coefficient is piecewise constant scalar and its
distribution is local quasi-monotone, it is shown
theoretically that the estimators developed in this paper are robust
with respect to the size of jumps. Numerical experiments are also
performed to support the theoretical results.
\end{abstract}

\section{Introduction}\label{intro}
\setcounter{equation}{0}

A posteriori error estimation for finite element methods has been
extensively studied for the past three decades (see, e.g., books by
Verf\"urth \cite{Ver:96, Ver:13}, Ainsworth and Oden \cite{AiOd:00},
Babu\v{s}ka and Strouboulis \cite{BaSt:01}, and references therein).
The widely adapted estimator is probably the Zienkiewicz-Zhu (ZZ) recovery-based
error estimator \cite{ZiZh:87, ZiZh:92} due to its easy implementation, generality, and
ability to produce quite accurate estimations. By first recovering 
a gradient in the conforming $C^0$
linear vector finite element space from the numerical gradient,
the ZZ estimator is defined as the $L^2$ norm of the difference 
between the recovered and the numerical gradients. 

Despite popularity of the ZZ estimator, 
it is also well known that the ZZ estimator over-refines regions where there are no error, and hence, they
fail to reduce the global error. This is shown by Ovall in
\cite{Ova:06b} through some interesting and realistic examples. 
Such a failure is
simply caused by using continuous functions (recovered gradient/flux)
to approximate discontinuous functions (true gradient/flux) in the
recovery procedure. By recovering flux and/or gradient in the 
respective $H(\divvr;\O)$  and $H(\curllr; \O)$ conforming finite element spaces,
in \cite{CaZh:09, CaZh:10a}, 
we developed and studied robust recovery-based implicit and explicit error estimators  for 
various lowest order finite element
approximations to the interface problems. The implicit error
estimator requires solution of a global $L^2$ minimization problem, and the explicit
error estimator uses a simple edge average.

The explicit recovery introduced in \cite{CaZh:09, CaZh:10a} is limited to the 
Raviart-Thomas ($RT$) \cite{BrFo:91} and  the first type of N\'ed\'elec ($\NE$)  \cite{Ned:80} elements 
of the lowest order for the respective flux and gradient recoveries. 
This simple averaging approach may not be extended to the 
Brezzi-Douglas-Marini ($B\!D\!M$) \cite{BrFo:91} and the second type of N\'ed\'elec  \cite{Ned:86} ($N\!D$) elements
of the lowest order  and to the diffusion problem with full coefficient tensor.
The purpose of this paper is first to introduce a general approach for
constructing explicit recovery of the flux/gradient for various lower order finite element  approximations to 
the diffusion problem with the full coefficient tensor. The approach, similar to 
\cite{CaZh:11}, is to localize the 
implicit recovery through a partition of the unity. 
For various lower order elements, we are able to reduce the local patch problem
to the edge/face patch which contains at most two elements. Hence, by solving 
a local minimization problem on this
two-element patch, we explicitly recover 
the flux/gradient. We then define the corresponding
estimators and establish their reliability and efficiency. 
When the diffusion coefficient is piecewise constant and
its distribution is local quasi-monotone, we are able to show 
theoretically that these estimators are robust
with respect to the size of jumps. For a benchmark test problem, whose coefficient
is not local quasi-monotone, numerical results also show the robustness of the estimators. 

For the conforming finite element approximation to the interface problem, robust error estimators have been studied by Bernardi and Verf\"urth \cite{BeVe:00} and Petzoldt \cite{Pet:02} for the residual-based estimator, Luce and Wohlmuth \cite{LuWo:04} for an equilibrated estimator on a dual mesh, and by us \cite{CaZh:09} for the recovery-based error estimator. Ainsworth in \cite{Ain:05, Ain:07} studied robust error estimators for nonconforming and mixed methods, respectively. Robust error estimators for locally conserved methods were studied by Kim \cite{Kim:07}. Recently, we studied robust recovery-based estimators for lowest order nonconforming, mixed, and discontinuous Galerkin methods (see \cite{CaZh:10a, CaYeZh:11}) via the $L^2$ recovery and for higher-order conforming elements in \cite{CaZh:10b} via a weighted $H(\divvr)$ recovery. Robust equilibrated residual error estimator are constructed by us in \cite{CaZh:11}. For interface problems with flux jumps, we studied robust residual- and recovery-based error estimators in \cite{CaZh:10c}.

The paper is organized as follows. Section 2 describes the diffusion
problem and its variational forms. Conforming, mixed, and nonconforming finite 
element methods are presented
in Section 3. Section 4 introduces the explicit recoveries
of the flux/gradient for those finite element approximations. The corresponding 
a posteriori error
estimators are introduced in Section 5 and their reliability and
efficiency bounds are established in Section 6. Finally, Section 7 provides
numerical results for a benchmark test problem.

\section{Diffusion Problem and Variational Form}\label{problems}
\setcounter{equation}{0}

Let $\O$ be a bounded polygonal domain in $\Re^2$,
with boundary $\p \O = \bar{\Gamma}_\sD \cup\bar{ \Gamma}_\sN$,
$\Gamma_\sD\cap \Gamma_\sN = \emptyset$,  and $\mbox{measure}\,(\Gamma_D)\not= 0$,
and let $\bn$ be the outward unit vector normal to the boundary.
Consider diffusion equation
\begin{equation}\label{pde}
    -\gradt (A(x) \grad u)  =  f   \quad\mbox{in} \quad  \O\\
\end{equation}
with boundary conditions
\beq\label{bc}
    -A\grad u \cdot \bn = g_\sN \quad\mbox{on} \quad \Gamma_\sN \quad
    \mbox{and}\quad u = g_\sD \quad\mbox{on}\quad \Gamma_\sD.
\eeq 
For simplicity of presentation,  assume that $f \in L^2(\O)$, that $g_\sD$  and
$g_\sN$ are piecewise affine functions and constants, respectively,
and that $A$ is a symmetric, positive definite piecewise constant
matrix.

Here and thereafter, we use standard notations and
definitions for the Sobolev spaces. Let
 \[
 H^1_{g,D}(\O)=\{v\in H^1(\O)\,|\, v=g_\sD\mbox{ on }\Gamma_D\}
 \quad\mbox{and}\quad
 H^1_D(\O)=H^1_{0,D}(\O).
 \]
Then the corresponding variational problem is to  find $u \in
H^1_{g,D}(\O)$ such that
 \beq \label{vp}
    a(u,\,v)\equiv (A\grad u, \grad v)
    = (f, v) - ( g_\sN, v )_{\Gamma_\sN}\equiv  f(v)
    \quad \forall \; v\in H^1_D(\O),
 \eeq
where $(\cdot, \cdot)_{\omega}$ is the $L^2$ inner product on the domain $\o$.
The subscript $\omega$ is omitted when $\o=\O$.

In two dimensions, for $\btau= (\tau_1,\,\tau_2)^t$, define the
divergence and curl operators by
 \[
 \divv \btau:= \dfrac{\p\tau_1}{\p x_1} + \dfrac{\p\tau_2}{\p x_2}
 \quad\mbox{and}\quad
 \curlt \btau := \dfrac{\p\tau_2}{\p x_1} - \dfrac{\p\tau_1}{\p
 x_2},
 \]
respectively. For a scalar-valued function $v$, define the operator
$\gperp$ by
$$
\gperp v: = (\dfrac{\p v}{\p x_2},\,- \dfrac{\p v}{\p
x_1})^t.
$$
We shall use the following Hilbert spaces
\begin{eqnarray*}
 && H(\divvr;\O)=\{\btau\in L^2(\O)^2 |\,\gradt\btau\in L^2(\O)\}\\[2mm]
  \mbox{ and} &&
 H(\curll;\O)=\{\btau\in L^2(\O)^2 |\,\curlt
 \btau \in L^2(\O)\}
 \end{eqnarray*}
equipped with the norms
\[
\|\btau\|_{H(\divvr;\,\O)}=\left(\|\btau\|^2_{0,\O}+\|\gradt\btau\|^2_{0,\O}
 \right)^\frac12
 \quad\mbox{and}\quad
 \|\btau\|_{H(\curll;\,\O)}=\left(\|\btau\|^2_{0,\O}+\|\curlt\btau\|^2_{0,\O}
 \right)^\frac12,
\]
respectively. Let
 \begin{eqnarray*}
  H_{g,N}(\divvr;\O)\!
  &=&\{\btau\in H(\divvr;\O) |\,\btau\cdot \bn|_{\Gamma_\sN}= g_\sN\},\quad 
 H_N(\divvr;\O)\!=H_{0,N}(\divvr;\O),\\[2mm]
  \mbox{ and }\,\,
  H_D(\curll;\O)\!
  &=&\{\btau\in H(\curll;\O) |\,\btau\cdot
 \bt\big|_{\Gamma_\sD}=0\},
 \end{eqnarray*}
where $\bn = ( n_1,n_2)^t$ and $\bt =(t_1,t_2)^t  =
(-n_2,n_1)^t$ are the unit vectors outward normal to and
tangent to the boundary $\p\O$, respectively.

Define the flux by
 \[
 \bsigma = -A(x)\grad u \quad\mbox{in }\,\O,
 \]
then the mixed variational formulation is to find $(\bsigma,\,u)\in
H_{g,N}(\divvr;\O)\times L^2(\O)$  such that
\begin{equation}\label{mixed}
 \left\{\begin{array}{lclll}
 (A^{-1}\bsigma,\,\btau)-(\divv \btau,\, u)
 &=& -(\btau \cdot \bn, g_\sD)_{\Gamma_\sD}
  \quad & \forall\,\, \btau \in
 H_N(\divvr;\O),\\[2mm]
 (\divv \bsigma, \,v) &=& (f,\,v)&\forall \,\, v\in L^2(\O).
\end{array}\right.
\end{equation}

\section{Finite Element Approximation}\label{FEMs}
\setcounter{equation}{0}

\subsection{Finite Element Spaces}

For simplicity, consider only triangular elements.
Let $\cT=\{K\}$ be a regular triangulation of the domain $\O$,
and denote by $h_\sK$ the diameter of the element $K$.
We assume that $A$ is piecewise constant matrix on the mesh $\cT$.
Denote the set of all nodes of the triangulation by
$ \cN := \cN_{_I}\cup\cN_{_D}\cup\cN_{_N},
$ where $\cN_{_I}$ is the set of all interior nodes and
$\cN_{_D}$ and $\cN_{_N}$ are the sets of all boundary nodes
belonging to the respective $\overline{\Gamma}_D$ and $\Gamma_N$.
Denote the set of all edges of the triangulation by $
 \cE := \cE_{_I}\cup\cE_{_D}\cup\cE_{_N},
$ where $\cE_{_I}$ is the set of all interior element edges and
$\cE_{_D}$ and $\cE_{_N}$ are the sets of all boundary edges
belonging to the respective $\Gamma_D$ and $\Gamma_N$. 

For each $F \in \cE$, 
denote by
$\bn_\sF= (n_{1,\sF},n_{2,\sF})^t$ a unit vector normal to $F$;
then $\bt_\sF =
-(n_{2,\sF},n_{1,\sF})^t$ is a unit vector tangent to $F$. 
Let
$K_\sF^-$ and $K_\sF^+$ be two elements sharing the common edge $F$
such that the unit outward normal vector of $K_\sF^-$ coincide with
$\bn_\sF$. When $F\in \cE_\sD \cup \cE_\sN$, $\bn_\sF $ is the unit
outward vector normal to $\p\O$ and denote by $K_\sF^-$ the element
having the edge $F$.   
For interior edges $F\in \cE_\sI$, the selection of $\bn_\sF$ is arbitrary but globally fixed.
For a function $v$ defined on
$K_\sF^-\cup K_\sF^+$, denote its traces on $F$ by $v|_\sF^-$ and
$v|_\sF^+$, respectively. 
The jump over the edge $F$ is denoted by
$$
\jump{v}_\sF := \left\{
\begin{array}{llll}
v|_\sF^- - v|_\sF^+ & F\in \cE_\sI, \\
v|_\sF^- & F\in\cE_\sD \cup \cE_\sN.
\end{array}
\right.
$$
(When there is no ambiguity, the  subscript or superscript $F$ in
the designation of jump and other places will be dropped.) 
%

For each $K\in\cT$, let
$P_k(K)$ be the space of polynomials of degree $k$. Denote the linear 
conforming and nonconforming (Crouzeix-Raviart) finite
element spaces \cite{Cia:78, GiRa:86} associated with the
triangulation $\cT$ by
 \[\begin{array}{ll}
  &S = \{v\in H^1(\O)\,\big|\,v|_K\in
 P_1(K)\quad\forall\,\,K\in\cT\}\\[2mm]
 and 
 & S^{nc}= \{v\in L^2(\O)\,\big|\,v|_K\in
 P_1(K)\,\,\forall\,\,K\in\cT,
 \mbox{ and } v \,\mbox{is continuous at}\,m_\sF  \;\forall\, F \in  \cE_\sI\},
 \end{array}\]
respectively. Let
 \[
 \begin{array}{ll}
S_{g,\sD}= \{v\in S |\,v=g_\sD \,\,\mbox{on}\, \, \Gamma_\sD\},
& \quad S^{nc}_{g,\sD}=\{v\in S^{nc}\,|\, v(m_\sF) =g_\sD(m_\sF) \,\,\forall\,\, F\in \cE_\sD\},\\[2mm]
 S_{\sD}= \{v\in S |\,v=0 \,\,\mbox{on}\, \, \Gamma_\sD\}, 
 & \quad S^{nc}_{\sD}=\{v\in S^{nc}\,|\, v(m_\sF) =0 \,\,\forall\,\, F\in \cE_\sD\}.
\end{array} \]

The $\Hdiv$ conforming Raviart-Thomas (RT) and Brezzi-Douglas-Marini
(BDM) spaces \cite{BrFo:91} of the lowest order are defined by
 \begin{eqnarray*}
 && RT=\{\btau\in H(\divvr;\,\Omega)\big|\,
 \btau|_K\in RT(K)\,\,\forall\,K\in\cT\}\\[2mm]
 \mbox{and} &&
 \BDM=\{\btau\in H(\divvr;\,\Omega)\big|\,
 \btau|_K\in BDM(K)\,\,\,\,\forall\,\,K\in\cT\},
 \end{eqnarray*}
 respectively, where $RT(K)=P_0(K)^2 +(x_1,x_2)^t\,P_0(K)$ and $\BDM(K)=P_1(K)^2$.
The $\Hcrl$-conforming first \cite{Ned:80} 
and second \cite{Ned:86} types of N\'ed\'elec
spaces of the lowest order are defined by
 \begin{eqnarray*}
&& \NE\!=\{\btau\in H(\curll;\Omega)\big|\,
 \btau|_K\in \NE(K)\,\forall\,K\in\cT\} \\[2mm]
  \mbox{and} &&
 \ND=\{\btau\in H(\curll;\Omega)\big|\,
 \btau|_K\in  \ND(K)\,\,\forall\,K\in\cT\},
 \end{eqnarray*}
 respectively, where $\NE(K)\!=\!P_0(K)^2+(x_2,-x_1)^tP_0(K)$ and $\ND(K) =P_1(K)^2$.
For convenience, denote $RT(K)$ and $\BDM(K)$ by $\cV(K)$,
$RT$ and $\BDM$ by $\cV$, $\NE(K)$ and $\ND(K)$ by $\cW(K)$, and $\NE$
and $\ND$ by $\cW$. Also, let
$$
P_0 =\{v\in L^2(\O)\,\big|\, v|_K \in P_0(K) \,\, \forall\,\, K\in
\cT\}.
$$
Definitions and properties of bases for  the $RT$, $BDM$,
$\NE$, and $\ND$ spaces  on an element $K$ are presented
in Appendix A.

Finally, we define the discrete gradient, divergence, and curl
operators by
$$
 (\grad_h v)|_K := \grad(v|_K),
 \quad (\grad_h\cdot \btau)|_K := \divv(\btau|_K),
 \quad\mbox{and}\quad
 (\grad_h\times \btau)|_K := \curlt(\btau|_K)
$$
for all $K\in \cT$, respectively.

\subsection{Finite Element Approximation}

The conforming finite element method is to seek $u_{c} \in
S_{g,\sD}$ such that
\begin{eqnarray} \label{problem_c}
  ( A\grad u_{c},\, \grad v) &=& (f,v)  \qquad \,\forall\, v\in
  S_{\sD},
\end{eqnarray}
the mixed finite element method is to seek $(\bsigma_{m},u_m) \in
\left(\cV\cap H_{g,\sN}(\divvr;\,\O)\right)\times P_0$ such that
\begin{equation}\label{problem_mixed}
 \left\{\begin{array}{lclll}
 (A^{-1}\bsigma_m,\,\btau)-(\divv \btau,\, u_m)&=& -(\btau\cdot\bn, g_\sD)_{\Gamma_\sD}
 \quad & \forall\,\, \btau \in \cV \cap H_N(\divvr;\O),\\[2mm]
 (\divv \bsigma_m,\, v) &=& (f,\,v)&\forall \,\, v\in P_0,
\end{array}\right.
\end{equation}
and the nonconforming finite element method is to find $u_{nc} \in
S^{nc}_{g,\sD}$ such that 
\begin{eqnarray} \label{problem_nc}
  ( A\grad_h u_{nc},\, \grad_h v) &=& (f,v)  \qquad \,\forall\, v\in
  S^{nc}_{\sD}.
\end{eqnarray}

\section{Explicit Flux and Gradient Recoveries}
\setcounter{equation}{0}

In \cite{CaZh:09, CaZh:10a, CaYeZh:11}, we studied flux and/or gradient recoveries
for various lower order finite element approximations to the diffusion problem. A unique feature
of those recoveries is that the recovered quantities are in proper finite element spaces. However, 
those recoveries require solutions of global problems with mass matrices. In this section,
we introduce explicit recovery procedures. This is done by first decomposing 
the error of the flux/gradient through a partition of the unity
and then approximating the flux/gradient error by local patch problems. 
The partition of the unity is based on nodal basis functions of 
the non-conforming linear element, and hence the local patch problems contain 
at most two elements.

\subsection{Explicit Flux Recovery for Conforming Method}

Let $u_c$ be the conforming linear finite element approximation defined in 
(\ref{problem_c}). Denote by 
  \[
\hat{\bsigma}_c=-A\grad u_c, \quad e_c = u-u_c, \quad\mbox{and}\quad 
\bE_c=\bsigma-\hat{\bsigma}_c = -A\grad e_c,
 \]
the numerical flux, the solution error, and the flux error, respectively. 
In this section, we introduce an explicit flux recovery procedure. 
This will be done through approximating the error flux $\bE_c$
by local patch problems. 

To this end, let $\phi_{\sF}^{nc}(\bx) \in S^{nc}$ be the nodal
basis function of the linear nonconforming element associated with the
edge $F \in \cE$. Denote by 
 \[\o_{\sF} =
\mbox{supp}(\phi_{\sF}^{nc}(\bx))
 \]
the support of $\phi_{\sF}^{nc}$,
which contains either two or one triangles for the respective interior 
or boundary edges. 
Denote the collection of triangles in $\o_{\sF}$ by
 \[
 \cT_\sF = \{ K\in \cT: \o_\sF \cap K \neq \emptyset \}.
 \]
Let $\cE_{b,\sF}$ be the collection of the boundary edges of
$\o_\sF$ that does not contain the edge $F$. Then
 the collection of edges of triangles in $ \cT_\sF$ is given by
$$
\cE_\sF = \{ E \in \cE : E \cap \overline{\o}_\sF \neq \emptyset\} =
\{F\} \cup \cE_{b,\sF} \quad \forall\,\, F \in \cE.
$$
It is also easy to check that 
 \begin{equation}\label{non-basisf}
 \phi_\sF^{nc}(\bx) \equiv 1\quad\mbox{on }\,\, F
 \quad\mbox{and}\quad
 \int_{E} \phi_\sF^{nc}\,ds =0
 \quad\forall\,\, E \in \cE_{b,\sF}.
 \end{equation}
The set of functions $\{\phi_{\sF}^{nc}\}_{F\in\cE}$ forms a partition
of the unity in $\O$:
\[
    \sum_{F\in \cE} \phi_\sF^{nc}(\bx) \equiv 1 \quad \forall\, \bx\in
    \O,
\]
which leads to the following decomposition of the error flux:
 \[
 \bE_c= \sum_{F\in \cE} \big(\phi_{\sF}^{nc}\,\bE_c\big)
 = \sum_{F\in \cE} \big(-\phi_{\sF}^{nc} A\grad
 e_c\big).
 \]
 
On edge $F\in \cE_\sI\cup\cE_\sN$, denote the normal components of the numerical flux by 
\beq\label{n-flux}
 \hat{\sigma}^+_{c,\sF}=\big(\hat{\bsigma}_c|_{K^+_\sF} \cdot\bn_\sF\big)|_{\sF}
 \quad\mbox{and}\quad
 \hat{\sigma}^-_{c,\sF}=\big(\hat{\bsigma}_c|_{K^-_\sF} \cdot\bn_\sF\big)|_{\sF}
 \eeq
and the jump of the numerical flux  by
\[
 j^c_{f,\sF}  
    \equiv\jump{\hat{\bsigma}_c\cdot \bn_\sF}_\sF    
     = \left\{\begin{array}{lll}
   \hat{\sigma}^-_{c,\sF}- \hat{\sigma}^+_{c,\sF}, & \forall\,\, F\in \cE_\sI,\\[3mm]
      \hat{\sigma}^-_{c,\sF}-g_\sN,  &\forall\,\, F\in \cE_\sN.
 \end{array}
 \right.
 \]

By the first equality in (\ref{non-basisf}) and the continuity of the normal component of the true flux,
it is easy to see that the jump of the normal component of the local error flux 
$\phi_{\sF}^{nc} \,\bE_c$ on edge $F\in \cE_\sI\cup\cE_\sN$ satisfies
 \beq\label{c-jump}  
  \jump{\phi_\sF^{nc}\,\bE_c \cdot \bn_\sF}_\sF  =- j^c_{f,\sF} .
    \eeq
 Note that 
\[
 \phi_\sF^{nc}\,\bE_c \cdot \bn_\sF
 =\left(-\phi_\sF^{nc}A \grad e_c\right) \cdot \bn_\sE \approx 0
  \quad\mbox{on }\,\,  E \in \cE_{b, \sF}.
 \]
Therefore, we introduce the following approximation to the local error 
flux $\phi_\sF^{nc}\,\bE_c$ on the local patch $\o_\sF$:
\begin{itemize}
\item[(1)] for every $F\in \cE_\sD$, set 
 \beq\label{rt-D}
 \sd_{c,\sF} =0 \quad \mbox{on }\,\, K^-_\sF;
 \eeq
 
 \item[(2)] for  every edge $F\in \cE_\sI\cup\cE_\sN$, find $\sd_{c,\sF} \in \cV^c_{_{-1, F}}$ such that
 \beq \label{c-mini}
 \| A^{-1/2}\sd_{c, \sF}  \|_{0,\o_\sF} =
 \min_{\btau \in \cV^c_{_{-1,F}}} \|A^{-1/2}\btau \|_{0,\o_\sF},
 \eeq
where $\cV^c_{_{-1, F}}$ with $\cV= RT$ or $BDM$  is a local 
finite element space defined by
 \[
 \cV^c_{_{-1, F}}
 = \{ \btau\in\! L^2(\omega_\sF\!) \big|\, \btau|_\sK\!\!\in\!\cV(\!K)\,\forall\, K\in \cT_\sF,
\,  \jump{\btau\cdot\bn_\sF}_\sF\!\!
 =\! -j^c_{f,\sF},\,
 \, \btau|_\sE\cdot\bn_{_E}\! = 0\,
 \forall \,E \in\! \cE_{b, \sF}\!\}.
 \]
\end{itemize}

With the approximations defined in (\ref{rt-D}) and (\ref{c-mini}), 
the global approximation to the error flux is then defined by
 \beq\label{app-err-f}
    \sd_c =  \sum_{F\in\cE}  \sd_{c,\sF}
    = \sum_{F\in\cE_\sI}  \sd_{c,\sF}+\sum_{F\in\cE_\sN}  \sd_{c,\sF}.
 \eeq
This yields the following recovered flux for the conforming linear element:
 \beq\label{app-f}
 \bsigma_c = \sd_c +\hat{\bsigma}_c\in H(\divvr;\O).
 \eeq
The fact that $\bsigma_c  \in H(\divvr;\O)$ follows from (\ref{c-jump}) that
 \[
 \jump{\bsigma_c\cdot\bn_\sF}_\sF
 = \jump{(\sd_c+ \hat{\bsigma}_c)\cdot\bn_\sF}_\sF 
 =- j^c_{f,\sF}+ \jump{\hat{\bsigma}_c\cdot\bn_\sF}_\sF=0
  \quad\mbox{on }\,\, F\in\cE_\sI.
  \]

\subsubsection{Solution of (\ref{c-mini})}

The recovered flux defined in (\ref{app-f}) requires solutions of the local problems defined in 
(\ref{c-mini}), which are constrained minimization problems. 
This section studies solutions of (\ref{c-mini}).
 
 
To this end, let $\bphi^{rt}_\sF$ be the local $RT$ basis function
given in (\ref{rt}) in Appendix A, define the global $RT$ basis function associated 
with the edge $F$ by
 \[
 \bpsi^{rt}_{\sF}
 = \left\{\begin{array}{llll}
 \bphi^{rt}_\sF|_{K_\sF^-}, & \bx \in K_\sF^-,  \\[4mm]
 -\bphi^{rt}_\sF|_{K_\sF^+},  & \bx \in K_\sF^+,  \\[4mm]
 0, & \bx \not\in \o_\sF, 
 \end{array}
 \right. \forall\,\, F\in\cE_\sI
 \,\mbox{ and }\,
 \bpsi^{rt}_{\sF}
 = \left\{\begin{array}{llll}
 \bphi^{rt}_\sF|_{K_\sF^-}, & \bx \in K_\sF^-,  \\[4mm]
 0, & \bx \not\in \o_\sF, 
 \end{array}
 \right. \forall\,\, F\in \cE_\sD\cup\cE_\sN .
 \]
for any $F\in\cE_\sI$ and by \[ \bpsi^{rt}_{\sF}  = \left\{\begin{array}{llll}
 \bphi^{rt}_\sF|_{K_\sF^-}& \bx \in K_\sF^-,  \\[4mm]  0& \bx \not\in \o_\sF
 \end{array} \right.  \] for any $F\in\cE_\sD\cup\cE_\sN$. 
Denote by $\bpsi^{rt,-}_{\sF}$ and
$\bpsi^{rt,+}_{\sF}$ the restriction of $\bpsi^{rt}_{\sF}$ on
$K_\sF^-$ and $K_\sF^+$, respectively. 
To solve (\ref{c-mini}), set 
 \beq\label{sigmajN}
 \bsigma^c_{j,\sF}=-j^c_{f,\sF}  \bpsi_\sF^{rt,-}
 \quad\mbox{on }\,\, K_\sF^-
 \eeq
for any $F\in\cE_\sN$ and set 
 \beq\label{sigmajI}
 \bsigma^c_{j,\sF} = \left\{\begin{array}{lll}
 -j^c_{f,\sF}  \bpsi_\sF^{rt,-}, & \mbox{on} & K_\sF^-, \\[4mm]
 0,  & \mbox{on} & K_\sF^+
 \end{array}\right.
 \eeq
for any $F\in \cE_\sI$. By (\ref{proprt}), it is easy to check that for $F\in\cE_\sI\cup\cE_\sN$
 \[
 \jump{\bsigma^c_{j,\sF} \cdot\bn_\sF}_\sF= - j^c_{f,\sF}
 \quad\mbox{and}\quad
 \bsigma^c_{j,\sF} \cdot\bn_\sE=0\quad\mbox{on}\,\,E\in\cE_{b,\sF}.
 \]
Hence, for any Neumman boundary edge $F\in \cE_\sN$, we have 
 \beq\label{rt-N}
 \sd_{c,\sF} =\bsigma^c_{j,\sF}
 \quad\mbox{on }\,\, K_\sF^-,
 \eeq
and for any interior edge $F\in\cE_\sI$, we have
 \[
  \sd_{c,\sF} - \bsigma^c_{j,\sF} \in H(\divvr;\o_\sF)
 \quad\mbox{and}\quad
 \left(\sd_{c,\sF} - \bsigma^c_{j,\sF}\right)|_\sE \cdot \bn_\sE=0\quad\forall\,\, E\in \cE_{b,\sF}.
 \]
Let 
 \[
 \st_{c,\sF} =  \sd_{c,\sF} - \bsigma^c_{j,\sF},
 \quad H_0(\divvr;\o_\sF)=\{ \btau\in H(\divvr;\o_\sF)\, |\,\btau\cdot\bn|_{\partial \o_\sF} = 0\},
 \]
 and
 \[
  \cV^c_{_{F}}
 = \{ \btau\in H_0(\divvr;\o_\sF) \,|\, \btau|_\sK\in  \cV(K)\,\,\forall\,\, K\in \cT_\sF
\},
 \]
then the minimization problem in (\ref{c-mini}) for $F\in \cE_\sI$ is equivalent to finding 
$\st_{c,\sF} \in \cV^c_{_{F}}$ such that 
  \[
  \| A^{-1/2}\left(\st_{c,\sF}  +\bsigma^c_{j,\sF}\right) \|_{0,\o_\sF} =
  \min_{\btau \in \cV^c_{_{F}}} \|A^{-1/2}\left(\btau+\bsigma^c_{j,\sF}\right) \|_{0,\o_\sF}.
  \]
The corresponding variational formulation is to find $\st_{c,\sF} \in \cV^c_{_{F}}$ such that 
 \beq \label{c-var}
 \left(A^{-1}\st_{c,\sF}  ,\,\btau\right)_{0,\o_\sF} =
 -\left(A^{-1}\bsigma^c_{j,\sF} ,\,\btau\right)_{0,\o_\sF}
 \quad\forall\,\,\btau\in \cV^c_{_{F}}.
 \eeq
Note that for any interior edge (\ref{c-var}) has either one (RT)
or two (BDM) unknowns. Their explicit formulas will be introduced in
the subsequent section. 

\subsubsection{Explicit Formula for Flux Recovery} 

This section derives explicit formulas for the solution of  (\ref{c-var}) 
and, hence, for the $RT$ and $\BDM$ recoveries. 

First, we consider the $RT$ recovery. Since 
$\st_{c,rt,\sF} \in RT^c_{_{F}}\subset H_0(\divvr;\o_\sF)$, we have
 \[
 \st_{c,rt,\sF} =\sd_{c,rt,\sF} -  \bsigma^c_{j,\sF}
 = a_{rt,\sF} j^c_{f,\sF} \bpsi_{\sF}^{rt}
 \quad\mbox{on}\,\,\omega_\sF
 \]
for all $F\in \cE_\sI$, which, together with (\ref{c-var}), yields 
 \[
 a_{rt,\sF}
 = \dfrac{ \beta_{rt,\sF}^- } { \beta_{rt,\sF}^- +\beta_{rt,\sF}^+  }
 \quad\mbox{with}\quad
 \beta_{rt,\sF}^\pm =  \left(A^{-1} \bpsi_{\sF}^{rt},\,\bpsi_{\sF}^{rt}\right)_{K_\sF^\pm}.
  \]
Hence, for any interior edge $F\in\cE_\sI$, we have
 \beq\label{rt-I}
 \sd_{c,rt,\sF} = \st_{c,rt,\sF} +  \bsigma^c_{j,\sF}
 = \left\{
 \begin{array}{lll}
  -\left(1-a_{rt,\sF}\right)
  j^c_{f,\sF} \bpsi_{\sF}^{rt,-},
    & \mbox{on} & K_\sF^-, \\[3mm]
 a_{rt,\sF}
    j^c_{f,\sF} \bpsi_{\sF}^{rt,+} , & \mbox{on} & K_\sF^+.
 \end{array}
 \right.
\eeq
Combining with (\ref{app-err-f}) and (\ref{rt-N}),
the global approximation to the error flux is given by
\beq\label{app-err-f-rt}
    \sd_{c,rt} 
    = \sum_{F\in\cE_\sI}  \sd_{c,rt,\sF}
    -\sum_{F\in\cE_\sN}  j^c_{f,\sF}  \bpsi_\sF^{rt,-}
 \eeq
 with $\sd_{c,rt,\sF}$ defined in (\ref{rt-I}).
 
 Since the numerical flux is a piecewise constant vector, it has the following 
 local representation on each element $K\in\cT$ (see Lemma 4.4 of \cite{CaZh:09}):
\[
\hat{\bsigma}_c|_\sK = \sum_{F\in \p K} \left( \hat{\bsigma}_c|_\sF\cdot \bn_\sK \right) 
\bphi_{\sF}^{rt},
\]
where $\bn_\sK$ is the unit outward vector normal to $\partial K$.  
Globally, for any interior edge $F\in\cE_\sI$, we have
 \beq\label{n-flux-rep}
 \hat{\bsigma}^c_{\sF}
 = \left\{
 \begin{array}{lll}
   \hat{\sigma}^-_{c,\sF}  \bpsi_{\sF}^{rt,-},
    & \mbox{on} & K_\sF^-, \\[5mm]
   \hat{\sigma}^+_{c,\sF}  \bpsi_{\sF}^{rt,+},
     & \mbox{on} & K_\sF^+.
 \end{array}
 \right.
 \eeq
Now, by (\ref{app-f}) and (\ref{app-err-f-rt}), the explicit formula for the recovered flux
using the RT element is then
 \beq\label{rt-explicit}
 \bsigma_{c}^{rt} = \sd_{c,rt}+\hat{\bsigma}_c = \sum_{F\in\cE}
 \sigma_{_{c,F}}^{rt}  \bpsi_\sF^{rt},
 \eeq
where the nodal value (i.e., the normal component of
$\bsigma_{c}^{rt}$ on the edge $F$), $\sigma_{_{c,F}}^{rt}$, is
given by
 \beq\label{rt-coef}
 \sigma_{_{c,F}}^{rt} = \left\{\begin{array}{llll}
 a_{rt,\sF} \hat{\sigma}^-_{c,\sF}
 +\left(1-a_{rt,\sF}\right)
 \hat{\sigma}^+_{c,\sF},
  & F\in\cE_\sI, \\[2mm]
   g_\sN, & F\in\cE_\sN,\\[2mm]
\hat{\sigma}^-_{c,\sF},  & F\in\cE_\sD
 \end{array}
 \right.
 \eeq
with the nodal values of the numerical fluxes, 
$\hat{\sigma}^+_{c,\sF}$ and $\hat{\sigma}^-_{c,\sF}$, defined 
in (\ref{n-flux}). Note that for any interior edge $F\in\cE_\sI$, the nodal value of the
recovered flux is an average of the numerical fluxes.

For interface problems, the recovered flux in (\ref{rt-explicit}) 
and the resulting estimator are similar to those introduced  and analyzed in \cite{CaZh:09}.
To this end, let $A|_\sK = \a_\sK I$  for any $K\in\cT$,  where $\a_\sK$ and $I$
are constant and the identity matrix, respectively.
Let 
 \[
 \a_\sF^-=\a_{K^-_\sF}
\quad{and}\quad\a_\sF^+=\a_{K^+_\sF}, 
\]
then
 \[
 \beta_\sF^- = \dfrac{1}{\a_\sF^-} \left(\bpsi_\sF^{rt},\,
 \bpsi_\sF^{rt}\right)_{K_\sF^-}
 \quad\mbox{and}\quad
 \beta_\sF^+ = \dfrac{1}{\a_\sF^+}
 \left(\bpsi_\sF^{rt}, \,\bpsi_\sF^{rt}\right)_{K_\sF^+}.
 \]
For a regular triangulation, the ratio of $\left(\bpsi_\sF^{rt},\,
\bpsi_\sF^{rt}\right)_{K_\sF^-} $ and $\left(\bpsi_\sF^{rt},\,
\bpsi_\sF^{rt}\right)_{K_\sF^+} $ are bounded above and below. Thus
 \beq\label{c-equi}
 a_{rt,\sF}=\dfrac{\beta_{\sF}^-}{\beta_{\sF}^- + \beta_{\sF}^+}
 \approx  \dfrac{\a_\sF^+}{\a_\sF^- + \a_\sF^+}
 \quad \mbox{and}\quad
 1-a_{rt,\sF}=\dfrac{\beta_{\sF}^+}{\beta_{\sF}^- + \beta_{\sF}^+} \approx
 \dfrac{\a_\sF^-}{\a_\sF^- + \a_\sF^+}.
 \eeq
(Here and thereafter, we will use $x \approx y$ to mean that there
exist two positive constants $C_1$ and $C_2$ independent of the mesh
size such that $C_1 x \leq y\leq C_2 x$.) (\ref{c-equi}) indicates
that the weights in the nodal values of the recovered flux may be
replaced by $\dfrac{\a_\sF^+}{\a_\sF^- + \a_\sF^+}$ and
$\dfrac{\a_\sF^-}{\a_\sF^- + \a_\sF^+}$, respectively. 

Next, we consider the $\BDM$ recovery.
For edge $F\in\cE$, let $\bs_\sF$ and $\be_\sF$ be endpoints of $F$ such that 
$\be_\sF-\bs_\sF =h_\sF\bt_\sF$. Let $\bphi^{bdm}_{_{s, F}}$ and $\bphi^{bdm}_{_{e, F}}$ be the two local $\BDM$ basis functions associated withe vertices $\bs_\sF$ and $\be_\sF$, respectively.
For $i=\{s,e\}$,
define the global $\BDM$ basis functions associated with
the edge $F$ by
 \[
\bpsi^{bdm}_{_{i, F}} \!=\! \left\{\!\!\begin{array}{llll}
\bphi^{bdm}_{_{i, F}}|_{K_\sF^-}, \!&\! \bx \in K_\sF^-,  \\[3mm]
-\bphi^{bdm}_{_{i,F}}|_{K_\sF^+}, \!&\! \bx \in K_\sF^+,  \\[3mm]
0, \!&\! \bx \not\in \o_\sF,
\end{array}
\right. \!\forall F\!\in\!\cE_\sI
\mbox{ and }
\bpsi^{bdm}_{_{i, F}}\! = \!\left\{\!\!\begin{array}{llll}
\bphi^{bdm}_{_{i, F}}|_{K_\sF^-}, \!&\! \bx \in K_\sF^-,  \\[4mm]
0, \!&\! \bx \not\in \o_\sF,
\end{array}
\right.\! \forall F\!\in \!\cE_\sD\!\!\cup\cE_\sN.
 \]
Denote by
$\bpsi^{bdm,-}_{_{i, F}}$ and $\bpsi^{bdm,+}_{_{i, F}}$
the restriction of $\bpsi^{bdm}_{_{i, F}}$ on $K_\sF^-$ and $K_\sF^+$,
respectively.

Again, since 
$\st_{c,bdm,\sF}=\sd_{c,bdm,\sF} -  \bsigma^c_{j,\sF}\in \BDM^c_{_{F}}\subset H_0(\divvr;\o_\sF)$, we have
 \[
 \st_{c,bdm,\sF}
  = a_{bdm,\sF} j^c_{f,\sF} \bpsi_{s,\sF}^{bdm} + b_{bdm,\sF} j^c_{f,\sF} \bpsi_{e,\sF}^{bdm}
 \quad \forall\,\, F\in \cE_\sI.
 \]
 For $i,\, j \in\{ s,\,e\}$, let
  \[
  \beta_{ij,\sF}^\pm =  \left(A^{-1} \bpsi_{i,\sF}^{bdm},\, \bpsi_{j,\sF}^{bdm}\right)_{K_\sF^\pm}
  \quad\mbox{and}\quad \beta_{ij,\sF}  = \beta_{ij,\sF}^- +\beta_{ij,\sF}^+ .
  \]
Solving (\ref{c-var}) yields
 \beq\label{BDM-a}
 a_{bdm,\sF}= \dfrac{ (\beta_{ss,\sF}^- + \beta_{se,\sF}^-)\beta_{ee,\sF}
 -(\beta_{se,\sF}^- +\beta_{ee,\sF}^- )\beta_{se,\sF}  }
 {\beta_{ss,\sF}\beta_{ee,\sF}- \beta_{se,\sF}^2}
 \eeq
 and
 \beq
 \label{BDM-b}
 b_{bdm,\sF}= \dfrac{ (\beta_{ss,\sF}^- + \beta_{se,\sF}^-)\beta_{ss,\sF}
 -(\beta_{se,\sF}^- +\beta_{ee,\sF}^- )\beta_{se,\sF}  }
 {\beta_{ss,\sF}\beta_{ee,\sF}- \beta_{se,\sF}^2}
 \eeq
Since $\bpsi_{\sF}^{rt,\pm}=\bpsi_{s,\sF}^{bdm,\pm}+\bpsi_{e,\sF}^{bdm,\pm}$, we have
  \beq\label{bdm-I}
  \sd_{c,bdm,\sF}=\st_{c,bdm,\sF}+\bsigma^c_{j,\sF}
  =\hat{a}_{bdm,\sF} j^c_{f,\sF}  \bpsi_{s,\sF}^{bdm}
 + \hat{b}_{bdm,\sF} j^c_{f,\sF}  \bpsi_{e,\sF}^{bdm} 
 \quad \forall\,\, F\in \cE_\sI
  \eeq
with
 \[
\hat{a}_{bdm,\sF} = \left\{ \begin{array}{llll}
 1- a_{bdm,\sF} ,  & \mbox{on}&K_{\sF}^-, \\[2mm]
a_{bdm,\sF}, & \mbox{on} & K_{\sF}^+,
 \end{array}
 \right.
 \quad\mbox{and}\quad
 \hat{b}_{bdm,\sF} = \left\{ \begin{array}{llll}
 1- b_{bdm,\sF} ,  & \mbox{on}&K_{\sF}^-, \\[2mm]
b_{bdm,\sF}, & \mbox{on} & K_{\sF}^+.
 \end{array}
 \right.
\]
Thus, by (\ref{app-err-f}) and (\ref{rt-N}) the error flux is given by
 \beq\label{app-err-f-bdm}
 \sd_{c,bdm}
 =\sum_{F\in\cE_\sI}  \sd_{c,bdm,\sF}
 -\sum_{F\in\cE_\sN}  j^c_{f,\sF}  (\bpsi_{s,\sF}^{bdm,-} +  \bpsi_{e,\sF}^{bdm,-}).
 \eeq
 with $\sd_{c,bdm,\sF}$ defined in (\ref{bdm-I}).
Now, by (\ref{app-f}) and (\ref{n-flux-rep}), the explicit formula for the recovered flux
using the $\BDM$ element is then
 \beq\label{bdm-explicit}
 \bsigma_{c}^{bdm} = \sd_{c,bdm}+\hat{\bsigma}_c =
 \sum_{F\in\cE} \sigma_{c,s,\sF}^{bdm} \bpsi_{s,\sF}^{bdm}
+ \sum_{F\in\cE} \sigma_{c,e,\sF}^{bdm} \bpsi_{e,\sF}^{bdm},
 \eeq
where
\[
 \sigma_{c,s,\sF}^{bdm} =
 \left\{
 \begin{array}{lllll}
 a_{bdm,\sF} \hat{\sigma}^-_{c,\sF}
 +(1-a_{bdm,\sF}) \hat{\sigma}^+_{c,\sF},
  & F\in\cE_\sI,\\[2mm]
  g_\sN,  &F\in \cE_\sN,\\[2mm]
   \hat{\sigma}^-_{c,\sF},  & F\in\cE_\sD.
 \end{array}
 \right.
 \]
 and
 \[
  \sigma_{c,e,\sF}^{bdm} =
 \left\{
 \begin{array}{lllll}
 b_{bdm,\sF} \hat{\sigma}^-_{c,\sF}
 +(1-b_{bdm,\sF}) \hat{\sigma}^+_{c,\sF},
  & F\in\cE_\sI,\\[2mm]
  g_\sN,  &F\in \cE_\sN,\\[2mm]
   \hat{\sigma}^-_{c,\sF},  & F\in\cE_\sD.
 \end{array}
 \right.
 \]
Note that, for any interior edge $F\in\cE_\sI$, the coefficients of the recovered flux are again weighted averages of the numerical fluxes.

For interface problems $A|_\sK = \alpha_\sK I$ with a regular triangulation, 
by a careful calculation, we can show that
 \beq\label{c-equi-bdm}
 a _{bdm,\sF} \approx  \dfrac{\a_\sF^+}{\a_\sF^- + \a_\sF^+}
 \approx b _{bdm,\sF}
 \quad \mbox{and}\quad
 1-a_{bdm,\sF} \approx
 \dfrac{\a_\sF^-}{\a_\sF^- + \a_\sF^+}
 \approx 1-b_{bdm,\sF}.
 \eeq

\subsection{Explicit Gradient Recovery for Mixed Method}

This section introduces an explicit gradient recovery based on the mixed 
finite element approximation in (\ref{problem_mixed}). Since derivation
is similar to that in the previous section, we briefly describe the recovery 
procedure and present an explicit 
formula of the recovered gradient.

Let $(\bsigma_m, u_m)$ be the solution of (\ref{problem_mixed}). Denote
by 
 \[
 \hat{\brho}_m= -A^{-1}\bsigma_m
 \quad\mbox{and}\quad
 \bE_m=\bsigma-\bsigma_m=-A\left(\grad u -\hat{\brho}_m\right)
 \] 
the numerical gradient and the flux error, respectively.
Then the gradient error is given by
 \beq\label{g-error}
 \grad u -\hat{\brho}_m=-A^{-1}\bE_m.
 \eeq
 Denote the tangential components of the numerical gradient on edge $F\in\cE$ by
  \beq\label{g-t}
  \hat{\rho}^+_{m, \sF}=\left(\hat{\brho}_m|_{K^+_\sF} \cdot\bt_\sF\right)|_{\sF}
  \quad\mbox{and}\quad
  \hat{\rho}^-_{m, \sF}=\left(\hat{\brho}_m|_{K^-_\sF} \cdot\bt_\sF\right)|_{\sF}
  \eeq
 and the edge jump of the numerical gradient on edge $F\in \cE_\sI\cup\cE_\sD$ by
\[
j^m_{g,\sF} \equiv \jump{ \hat{\brho}_m\cdot \bt_\sF}_\sF
=\left\{
\begin{array}{lll}
 \hat{\rho}^-_{m, \sF}-\hat{\rho}^+_{m, \sF}, & \forall \,\,F\in\cE_\sI, \\[2mm]
\hat{\rho}^-_{m, \sF}-\grad g_\sD\!\!\cdot\bt_\sF , &\forall \,\,F\in\cE_\sD.
\end{array}
\right.
\]
By the continuity of the tangential components of the true gradient, 
the edge jump of the tangential component of the local error gradient is given by
\beq\label{g-jump}
 \jump{-\phi_\sF^{nc} A^{-1} \bE_m \cdot \bt_\sF}_\sF
=-j^m_{g,\sF}.
\eeq

Since $\hat{\rho}^+_{m, \sF}$ and $\hat{\rho}^-_{m, \sF}$ are affine functions 
defined on $F\in\cE$, 
the $\ND$ element is needed for their approximations. To this end, as in the $\BDM$ case, 
for edge $F\in\cE$, let $\bs_\sF$ and $\be_\sF$ be endpoints of $F$ such that 
$\be_\sF-\bs_\sF =h_\sF\bt_\sF$,
let $\bphi^{nd}_{i, \sF}$ 
($i=s,\,e$) be the local $\ND$ basis functions given in (\ref{nd}) in Appendix A, and
define the global $\ND$ basis functions associated with edge $F$ by
 \[
 \bpsi^{nd}_{i, \sF} =\! \left\{\!\!\begin{array}{llll}
\bphi^{nd}_{i, \sF}|_{K_\sF^-},& \bx \in K_\sF^-,  \\[4mm]
-\bphi^{nd}_{i,\sF}|_{K_\sF^+}, & \bx \in K_\sF^+,  \\[4mm]
0,& \bx \not\in \o_\sF,
\end{array}
\right. \forall\,F\!\in\cE_\sI
\mbox{ and }
\bpsi^{nd}_{i, \sF} =\! \left\{\!\!\begin{array}{llll}
\bphi^{nd}_{i, \sF}|_{K_\sF^-},& \bx \in K_\sF^-,  \\[4mm]
0,& \bx \not\in \o_\sF,
\end{array}
\right. \forall\,F\! \in\cE_\sD\!\!\cup\cE_\sN.
 \]
Denote by $\bpsi^{nd,-}_{i, \sF}$ and $\bpsi^{nd,+}_{i, \sF}$ the restriction of
$\bpsi^{nd}_{i, \sF}$ on $K_\sF^-$ and $K_\sF^+$, respectively.

Denote  the tangential components of the numerical gradient, $\hat{\rho}^+_{m, \sF}$ and $\hat{\rho}^-_{m, \sF}$, at the endpoints by
 \[
 d^\pm_{s,\sF} = \hat{\rho}^\pm_{m,\sF}(\bs_\sF)    \quad\mbox{and}\quad
 d^\pm_{e,\sF} = \hat{\rho}^\pm_{m,\sF}(\be_\sF),
 \]
respectively. 
Then the numerical gradient has the following representation in local $\ND$ bases
$$
\hat{\brho}_m = 
 d^-_{s,\sF}\bpsi_{s,\sF}^{nd,-} -  d^-_{e, \sF}  \bpsi_{e,\sF}^{nd,-},\,\,  \mbox{   on  } \,\, K_\sF^-,
\quad\mbox{for}\quad F\in \cE_\sD
$$
and 
$$
\hat{\brho}_m = \left\{
\begin{array}{llll}
d^-_{s,\sF} \bpsi_{s,\sF}^{nd,-} - d^-_{e,\sF} \bpsi_{e,\sF}^{nd,-}, & \mbox{on} & K_\sF^-, \\[4mm]
d^+_{s,\sF} \bpsi_{s,\sF}^{nd,+} - d^+_{e,\sF} \bpsi_{e,\sF}^{nd,+}, & \mbox{on} & K_\sF^+,
\end{array}
\right.
\quad\mbox{for}\quad F\in \cE_\sI.
$$

Let 
 \[
 c_{s,\sF} =\left\{\begin{array}{ll}
 d^-_{s,\sF} - d^+_{s,\sF}, &  \forall \,\, F\in\cE_\sI,\\[4mm]
 d^-_{s,\sF} , &  \forall \,\, F\in\cE_\sD
 \end{array}\right.
   \quad \mbox{and }\,\,
 c_{e,\sF} = \left\{\begin{array}{ll}
   d^-_{e,\sF} - d^+_{e,\sF}, & \forall \,\, F\in\cE_\sI, \\[4mm]
   d^-_{e,\sF}, & \forall \,\, F\in\cE_\sD.
   \end{array}\right.
   \]
A simple calculation leads to 
\beq\label{jump-g}
j^m_{g,\sF} =
\left\{
\begin{array}{lll}
 c_{s,\sF} (\bpsi^{nd}_{s,\sF}\cdot\bt_\sF) - c_{e,\sF} (\bpsi^{nd}_{e,\sF}\cdot\bt_\sF), 
 & \forall \,\,F\in\cE_\sI, \\[4mm]
 c_{s,\sF} (\bpsi^{nd}_{s,\sF}\cdot\bt_\sF) - c_{e,\sF} (\bpsi^{nd}_{e,\sF}\cdot\bt_\sF)
-\grad g_\sD\!\!\cdot\bt_\sF , &\forall \,\,F\in\cE_\sD.
\end{array}
\right.
\eeq
Let 
 \beq\label{rhojD}
 \brho^m_{j,\sF}
    = - c_{s,\sF} \bpsi_{s,\sF}^{nd,-} + c_{e,\sF} \bpsi_{e,\sF}^{nd,-}, \quad\mbox{on }\,\, K_\sF^-
    \eeq
for $F\in\cE_\sD$, and let 
 \beq \label{rhojI}
\brho^m_{j,\sF} = \left\{
 \begin{array}{lll}
- c_{s,\sF} \bpsi_{s,\sF}^{nd,-} + c_{e,\sF} \bpsi_{e,\sF}^{nd,-},
& \mbox{on} & K_\sF^-, \\[4mm]
0,  & \mbox{on} & K_\sF^+
\end{array}
\right.
\eeq
for $F\in\cE_\sI$.
By the properties of the $\ND$ basis functions in (\ref{propnd0}) and  (\ref{propnd}),
it is easy to check that 
 \[
 \jump{\brho^m_{j,\sF}\cdot\bt_\sF}_\sF=-j^m_{g,\sF}
 \quad\mbox{and}\quad
 \brho^m_{j,\sF} \cdot\bt_\sE=0\quad\mbox{on}\,\,E\in\cE_{b,\sF}
 \]
 for $F\in \cE_\sI\cup\cE_\sD$. 
 
 Let 
 \[
 H_0(\mbox{curl};\o_\sF)=\{ \btau\in H(\mbox{curl};\o_\sF)\, |\,\btau\cdot\bt |_{\partial \o_\sF} = 0\},
 \]
 and let
 \[
  \ND^m_{\sF}
 = \{ \btau\in H_0(\mbox{curl};\o_\sF) \,|\, \btau|_\sK\in  \ND(K)\,\,\forall\,\, K\in \cT_\sF\}.
 \]
 In a similar fashion as that of the previous section,
  by (\ref{g-jump}), 
 we introduce the following approximation to the error gradient:
   \beq\label{app-err-g}
  \rd_{m,nd}=\sum_{F\in\cE_\sD} \rd_{m,nd,\sF}+\sum_{F\in\cE_\sI} \rd_{m,nd,\sF},
  \eeq
 where 
  \beq\label{app-err-g-F}
  \rd_{m,nd,\sF}=\left\{\begin{array}{ll}
  \brho^m_{j,\sF}, & F\in \cE_\sD,\\[2mm]
  \tilde{\brho}_{m,\sF} + \brho^m_{j,\sF}, & F\in \cE_\sI.
  \end{array}\right.
  \eeq
Here, $\tilde{\brho}_{m,\sF}\in \ND^m_{\sF}$ is the solution of the following minimization problem:
 \beq\label{mix-mini}
 \|A^{1/2}\left(\tilde{\brho}_{m,\sF}+\brho^m_{j,\sF}\right)\|_{0,\o_\sF}
 =\min_{\btau\in \ND^m_{\sF}} \|A^{1/2}\left(\btau+\brho^m_{j,\sF}\right)\|_{0,\o_\sF}.
 \eeq

Let
$$
\gamma_{ij,\sF}^- =  (A \bpsi_{i,\sF}^{nd},  \bpsi_{j,\sF}^{nd})_{K_\sF^-},\quad
\gamma_{ij,\sF}^+ =  (A \bpsi_{i,\sF}^{nd},  \bpsi_{j,\sF}^{nd})_{K_\sF^+},
\quad\mbox{and}\quad \gamma_{ij,\sF}  = \gamma_{ij,\sF}^- +\gamma_{ij,\sF}^+ 
$$
for $i,\, j \in\{ s,\,e\}$.
Solving (\ref{mix-mini}) leads to
$$
\tilde{\brho}_{m,\sF} = \rho_{s,\sF} \bpsi_{s,\sF}^{nd} + \rho_{e,\sF} \bpsi_{e,\sF}^{nd}
$$
with coefficients given by
\begin{eqnarray*}
\rho_{s,\sF} &=&\dfrac{\left(c_{s,\sF}\gamma_{ss,\sF}^- 
 -c_{e,\sF}\gamma_{se,\sF}^- \right)\gamma_{ee,\sF}
-
\left(c_{s,\sF}\gamma_{se,\sF}^- 
 -c_{e,\sF}\gamma_{ee,\sF}^- \right)\gamma_{se,\sF}
 }{\gamma_{ss,\sF}\gamma_{ee,\sF}- \gamma_{se,\sF}^2}\\
 &=&
 \dfrac{c_{s,\sF} \left(\gamma_{ss,\sF}^-\gamma_{ee,\sF} 
 - \gamma_{se,\sF}^- \gamma_{se,\sF} \right)
- c_{e,\sF}
\left(\gamma_{se,\sF}^- \gamma_{ee,\sF}
 -\gamma_{ee,\sF}^- \gamma_{se,\sF}\right)
 }{\gamma_{ss,\sF}\gamma_{ee,\sF}- \gamma_{se,\sF}^2}
 \end{eqnarray*}
and
\begin{eqnarray*}
\rho_{e,\sF} &=&\dfrac{\left(c_{s,\sF}\gamma_{se,\sF}^- 
 -c_{e,\sF}\gamma_{ee,\sF}^- \right)\gamma_{ss,\sF}
-
\left(c_{s,\sF}\gamma_{ss,\sF}^- 
 -c_{e,\sF}\gamma_{se,\sF}^- \right)\gamma_{se,\sF}
 }{\gamma_{ss,\sF}\gamma_{ee,\sF}- \gamma_{se,\sF}^2}\\
 &=&
 \dfrac{c_{s,\sF} \left(\gamma_{se,\sF}^-\gamma_{ss,\sF} 
 - \gamma_{ss,\sF}^- \gamma_{se,\sF} \right)
- c_{e,\sF}
\left(\gamma_{ee,\sF}^- \gamma_{ss,\sF}
 -\gamma_{se,\sF}^- \gamma_{se,\sF}\right)
 }{\gamma_{ss,\sF}\gamma_{ee,\sF}- \gamma_{se,\sF}^2}.
 \end{eqnarray*}
Hence, we have
 \beq\label{app-err-g-I}
\rd_{m,nd,\sF}=\tilde{\brho}_{m,\sF}+\brho^m_{j,\sF} = \left\{
\begin{array}{lllll}
(\rho_{s,\sF}  - c_{s,\sF})  \bpsi_{s,\sF}^{nd}
+ (\rho_{e,\sF} +c_{e,\sF})  \bpsi_{e,\sF}^{nd},&\mbox{on}& K_\sF^- ,\\[4mm]
\rho_{s,\sF}  \bpsi_{s,\sF}^{nd}
+ \rho_{e,\sF}  \bpsi_{e,\sF}^{nd},&\mbox{on}& K_\sF^+
\end{array}\right.
\eeq
for interior edge $F\in\cE_\sI$.
Now, the explicit formula for the recovered gradient using $\ND$ element is then
\beq
\brho^{nd}_m  =\rd_{m,nd}+\hat{\brho}_m = \sum_{F\in \cE} a^{nd}_{m,\sF}\bpsi^{nd}_{s,\sF}+
 \sum_{F\in \cE} b^{nd}_{m,\sF}\bpsi^{nd}_{e,\sF},
\eeq
where the coefficients are given by
$$
a^{nd}_{m,\sF} = \left\{\begin{array}{lllll}
\rho_{s,\sF} +d_{s,\sF}^+, & F\in \cE_\sI, \\[2mm]
d^-_{s,\sF},& F\in\cE_{\sN},\\[2mm]
\grad g_{\sD}\!\!\cdot\bt_\sF,& F\in\cE_{\sD}
\end{array}
\right.
\quad
\mbox{and}
\quad
b^{nd}_{m,\sF} = \left\{\begin{array}{lllll}
\rho_{e,\sF} -d_{e,\sF}^+, & F\in \cE_\sI, \\[2mm]
-d^-_{e,\sF},& F\in\cE_{\sN},\\[2mm]
-\grad g_{\sD}\!\!\cdot\bt_\sF,& F\in\cE_{\sD}.
\end{array}
\right.
$$
Notice that
$$
\rho_{s,\sF} +d_{s,\sF}^+ = \ell_{s,\sF} d^-_{s,\sF} + (1-\ell_{s,\sF})d^+_{s,\sF} 
-
\dfrac{ c_{e,\sF}
\left(\gamma_{se,\sF}^- \gamma^+_{ee,\sF}
 -\gamma_{ee,\sF}^- \gamma^+_{se,\sF}\right)
 }{\gamma_{ss,\sF}\gamma_{ee,\sF}- \gamma_{se,\sF}^2}
$$
with
$$
 \ell_{s,\sF}  =  \dfrac{ \left(\gamma_{ss,\sF}^-\gamma_{ee,\sF} 
 - \gamma_{se,\sF}^- \gamma_{se,\sF} \right) }
 {\gamma_{ss,\sF}\gamma_{ee,\sF}- \gamma_{se,\sF}^2},
$$
and
$$
\rho_{e,\sF} -d_{e,\sF}^+ = - \ell_{e,\sF} d^-_{e,\sF} - (1-\ell_{e,\sF})d^+_{e,\sF} 
+
 \dfrac{c_{s,\sF} \left(\gamma_{se,\sF}^-\gamma^-_{ss,\sF} 
 - \gamma_{ss,\sF}^- \gamma^+_{se,\sF} \right)
 }{\gamma_{ss,\sF}\gamma_{ee,\sF}- \gamma_{se,\sF}^2}.
 $$
with
$$
 \ell_{e,\sF}  =  \dfrac{ \left(\gamma_{ee,\sF}^-\gamma_{ss,\sF} 
 - \gamma_{se,\sF}^- \gamma_{se,\sF} \right) }
 {\gamma_{ss,\sF}\gamma_{ee,\sF}- \gamma_{se,\sF}^2}.
$$

Note that, for any interior edge $F\in\cE_\sI$, the coefficients of the recovered gradient are 
weighted averages of the numerical gradients plus some high order terms.

For interface problems $A|_\sK = \alpha_\sK I$ with a regular triangulation, 
by a careful calculation, we can show that
 \beq\label{c-equi-bdm}
 \ell _{s,\sF}
 \approx  \dfrac{\a_\sF^-}{\a_\sF^- + \a_\sF^+}
 \approx \ell _{e,\sF}
 \quad \mbox{and}\quad
 1-\ell_{s,\sF} \approx
 \dfrac{\a_\sF^+}{\a_\sF^- + \a_\sF^+}
 \approx 1-\ell_{e,\sF}.
 \eeq


\subsection{Explicit Flux and Gradient Recoveries for Nonconforming Method}

Let $u_{nc}$ be the solution of (\ref{problem_nc}). Denote by 
 \[
  \hat{\brho}_{nc}= \grad_h u_{nc}
 \quad\mbox{and}\quad
\hat{\bsigma}_{nc}= -A^{-1}\grad_h u_{nc}=-A^{-1} \hat{\brho}_{nc}
 \]
 the numerical gradient and the numerical flux, respectively.
 This section introduces explicit formulas of the recovered flux $\bsigma_{nc}\in H(\divvr, \O)$ 
 and the recovered gradient $\brho_{nc}\in H(\curll,\O)$  based on $\hat{\bsigma}_{nc}$
 and $ \hat{\brho}_{nc}$.
 Again, derivations are similar to those in the previous sections and, hence, descriptions in this section
 are brief. 
 
Denote the solution error, the flux error, and the gradient error by 
 \[
 e_{nc} = u-u_{nc},
  \quad 
 \bE_{nc} = \bsigma-\hat{\bsigma}_{nc} = -A\grad_h e_{nc},
  \mbox{ and }
  \grad u -\grad_h u_{nc}=\grad_h e_{nc}= -A^{-1}\bE_{nc},
  \]
respectively. Denote the normal components of the numerical flux on edge $F\in\cE$ by
  \beq\label{n-flux-nc}
 \hat{\sigma}^+_{nc, \sF}=\left(\hat{\bsigma}_{nc}|_{K^+_\sF} \cdot\bn_\sF\right)|_{\sF}
  \quad\mbox{and}\quad
 \hat{\sigma}^-_{nc, \sF}=\left(\hat{\bsigma}_{nc}|_{K^-_\sF} \cdot\bn_\sF\right)|_{\sF}
  \eeq
and the edge jump of the numerical flux by
 \[
 j^{nc}_{f,\sF} \equiv \jump{\hat{\bsigma}_{nc}\cdot\bn_\sF}_\sF
=\left\{
\begin{array}{lll}
 \hat{\sigma}^-_{nc, \sF}-\hat{\sigma}^+_{nc, \sF}, & \forall \,\,F\in\cE_\sI, \\[2mm]
\hat{\sigma}^-_{nc, \sF}-g_\sN , &\forall \,\,F\in\cE_\sN.
\end{array}
\right.
\]
Denote the tangential components of the numerical gradient on edge $F\in\cE$ by
  \beq\label{g-t-nc}
  \hat{\rho}^+_{nc, \sF}=\left(\hat{\brho}_{nc}|_{K^+_\sF}\cdot\bt_\sF\right)|_{\sF}
  \quad\mbox{and}\quad
  \hat{\rho}^-_{nc, \sF}=\left(\hat{\brho}_{nc}|_{K^-_\sF}\cdot\bt_\sF\right)|_{\sF}
  \eeq
and the edge jump of the numerical gradient by
 \[
j^{nc}_{g,\sF} \equiv \jump{\hat{\brho}_{nc}\cdot\bt_\sF}_\sF
=\left\{
\begin{array}{lll}
 \hat{\rho}^-_{nc, \sF}-\hat{\rho}^+_{nc, \sF}, & \forall \,\,F\in\cE_\sI, \\[2mm]
 \hat{\rho}^-_{nc, \sF}-\grad g_\sD\cdot\bt_\sF, &\forall \,\,F\in\cE_\sD.
\end{array}
\right.
\]
By the continuity of the true flux and true gradient, we have
\beq\label{jump-nc}
\jump{\phi_\sF^{nc} \bE_{nc} \cdot \bn_\sF}_\sF
=-j^{nc}_{f,\sF}
\,\,\mbox{ and }\,\,
\jump{-\phi_\sF^{nc} A^{-1} \bE_{nc} \cdot \bt_\sF}_\sF
=-j^{nc}_{g,\sF}.
\eeq

\subsubsection{ Explicit Formula for Flux Recovery}

In a similar fashion as in Section~4.1, the approximation to the error flux using the $RT$ element is given by
\beq\label{sdncrtf}
\sd_{nc,rt} = \sum_{F\in\cE_\sI}\sd_{nc,rt,\sF}
+ \sum_{F\in\cE_\sN}\sd_{nc,rt,\sF}
= \sum_{F\in\cE_\sI}\sd_{nc,rt,\sF}
-\sum_{F\in\cE_\sN}j^{nc}_{f,\sF}  \bpsi_\sF^{rt,-}
\eeq
with 
\beq\label{app-err-f-nc-F}
\sd_{nc,rt,\sF} 
= \left\{
 \begin{array}{lll}
 - \left(1-a_{rt,\sF}\right) j^{nc}_{f,\sF} \bpsi_{\sF}^{rt,-},
    & \mbox{on} & K_\sF^-, \\[3mm]
 a_{rt,\sF}
    j^{nc}_{f,\sF} \bpsi_{\sF}^{rt,+},  & \mbox{on} & K_\sF^+,
\end{array}
\right.
\eeq
where $a_{rt,\sF}$ is defined in Section 4.1.2.
Now, the explicit flux recovery using the
$RT$ element is given by 
\beq\label{rt-explicit-nc}
 \bsigma_{nc}^{rt} = \sd_{nc,rt}+\hat{\bsigma}_c 
 = \sum_{F\in\cE}
 \sigma_{_{nc,F}}^{rt}  \bpsi_\sF^{rt} \in H(\mbox{div},\O),
 \eeq
where the nodal value $\sigma_{_{nc,F}}^{rt}$ is
given by
 \beq\label{rt-coef-nc}
 \sigma_{_{nc,F}}^{rt} = \left\{\begin{array}{llll}
 a_{rt,\sF} \hat{\sigma}^-_{nc,\sF}
 +(1-a_{rt,\sF})\hat{\sigma}^+_{nc,\sF},
  & F\in\cE_\sI, \\[2mm]
   g_\sN, & F\in\cE_\sN,\\[2mm]
\hat{\sigma}^-_{nc,\sF},  & F\in\cE_\sD.
 \end{array}
 \right.
 \eeq

Using the $\BDM$ element, the approximation to the error flux is given by
 \begin{eqnarray}\nonumber 
 \sd_{nc,bdm}
 &=& \sum_{F\in\cE_\sI}\sd_{nc,bdm,\sF}
+ \sum_{F\in\cE_\sN}\sd_{nc,bdm,\sF}\\[2mm]\label{bdm-Ib}
 &=&\sum_{F\in\cE_\sI} j^{nc}_{f,\sF}\left(a_{bdm,\sF}   \bpsi_{s,\sF}^{bdm}
 + b_{bdm,\sF}   \bpsi_{e,\sF}^{bdm}\right)
 -\sum_{F\in\cE_\sN}  j^{nc}_{f,\sF}  (\bpsi_{s,\sF}^{bdm,-}+\bpsi_{e,\sF}^{bdm,-}),
 \end{eqnarray}
where $a_{bdm,\sF}$ and $b_{bdm,\sF}$ are defined in Section 4.1.2.
Now, the explicit flux recovery using the $\BDM$ element is given by 
 \beq\label{bdm-explicit-nc}
 \bsigma_{nc}^{bdm}  =
 \sum_{F\in\cE} \sigma_{nc,s,\sF}^{bdm} \bpsi_{s,\sF}^{bdm}
 +\sum_{F\in\cE} \sigma_{nc,e,\sF}^{bdm}  \bpsi_{e,\sF}^{bdm} \in H(\mbox{div},\O)
 \eeq
 where  where $\sigma_{nc,s,\sF}^{bdm}$ and $\sigma_{nc,e,\sF}^{bdm}$ is similar to $\sigma_{c,s,\sF}^{bdm}$
 and $\sigma_{c,e,\sF}^{bdm}$ defined in Section 4.1.2, i.e.,
\[
 \sigma_{nc,s,\sF}^{bdm} =
 \left\{
 \begin{array}{lllll}
 a_{bdm,\sF} \hat{\sigma}^-_{nc,\sF}
 +(1-a_{bdm,\sF}) \hat{\sigma}^+_{nc,\sF},
  & F\in\cE_\sI,\\[2mm]
  g_\sN,  &F\in \cE_\sN,\\[2mm]
   \hat{\sigma}^-_{nc,\sF},  & F\in\cE_\sD.
 \end{array}
 \right.
 \]
 and
 \[
  \sigma_{nc,e,\sF}^{bdm} =
 \left\{
 \begin{array}{lllll}
 b_{bdm,\sF} \hat{\sigma}^-_{nc,\sF}
 +(1-b_{bdm,\sF}) \hat{\sigma}^+_{nc,\sF},
  & F\in\cE_\sI,\\[2mm]
  g_\sN,  &F\in \cE_\sN,\\[2mm]
   \hat{\sigma}^-_{nc,\sF},  & F\in\cE_\sD.
 \end{array}
 \right.
 \]

 
\subsubsection{Explicit Formula for Gradient Recovery}

Let $\bphi^{ne}_\sF$ be the local $\NE$ basis function given in Appendix A, define 
the global $\NE$ basis function associated with the edge $F$ by
\[
\bpsi^{ne}_{F} =\! \left\{\!\!\begin{array}{llll}
\bphi^{ne}_\sF|_{K_\sF^-},& \bx \in K_\sF^-,  \\[2mm]
-\bphi^{ne}_\sF|_{K_\sF^+}, & \bx \in K_\sF^+,  \\[2mm]
0,& \bx \not\in \o_\sF,
\end{array}
\right. \forall F\in\cE_\sI
\mbox{ and }
\bpsi^{ne}_{F} =\! \left\{\!\!\begin{array}{llll}
\bphi^{ne}_\sF|_{K_\sF^-},& \bx \in K_\sF^-,  \\[2mm]
0,& \bx \not\in \o_\sF,
\end{array}
\right.
\forall F\in\cE_\sD\cup\cE_\sN.
\]
Denote by $\bpsi^{ne,-}_{F}$ and $\bpsi^{ne,+}_{F}$ the restriction of
 $\bpsi^{ne}_{F}$ on $K_\sF^-$ and $K_\sF^+$, respectively.
 Let
 \[
 a_{ne,\sF}
 = \dfrac{ \beta_{ne,\sF}^- } { \beta_{ne,\sF}^- +\beta_{ne,\sF}^+  }
 \quad\mbox{with}\quad
 \beta_{ne,\sF}^\pm =  \left(A^{-1} \bpsi_{\sF}^{ne},\,\bpsi_{\sF}^{ne}\right)_{K_\sF^\pm}.
  \]
Then the approximation to the gradient error is
\beq\label{app-err-g-nc}
\rd_{{nc},ne} =\sum_{F\in\cE_\sI}\rd_{{nc},ne,\sF}
+\sum_{F\in\cE_\sD}\rd_{{nc},ne,\sF}
= \sum_{F\in\cE_\sI}\rd_{{nc},ne,\sF}
+\sum_{F\in\cE_\sD}j^{nc}_{g,\sF}  \bpsi_\sF^{ne,-}
\eeq
with 
\beq\label{local-rho-nc}
\rd_{{nc},ne,\sF} 
= \left\{
 \begin{array}{lll}
  -\left(1-a_{ne,\sF}\right)
  j^{nc}_{g,\sF} \bpsi_{\sF}^{ne,-}, & \mbox{on} & K_\sF^-, \\[3mm]
 a_{ne,\sF}
     j^{nc}_{g,\sF} \bpsi_{\sF}^{ne,+},  & \mbox{on} & K_\sF^+.
\end{array}
\right.
\eeq
Now, the explicit gradient recovery using the $\NE$ element is given by 
\beq\label{ne-explicit-nc}
\brho_{nc,\sF}^{ne} = \rd_{nc,ne} +\hat{\brho}_{nc}
 = \sum_{F\in\cE} \rho_{nc,\sF}^{ne}  \bpsi_\sF^{ne} \in H(\mbox{curl},\O),
\eeq
where the nodal value $\rho_{{nc,F}}^{ne}$ is
given by
  \beq\label{ne-coef-nc}
 \rho_{nc,\sF}^{ne} 
 = \left\{ \begin{array}{lll}
 a_{ne,\sF}\hat{\rho}^-_{nc,F}
 +\left(1- a_{ne,\sF}\right)\hat{\rho}^+_{nc,F},
 & F\in\cE_\sI, \\[3mm]
 \grad g_\sD\!\!\cdot \bt_\sF, & F\in\cE_\sD, \\[3mm]
 \hat{\rho}^-_{nc,F}, & F\in\cE_\sN.
\end{array}
\right.
\eeq

Next, we describe the recovered gradient using the $\ND$ element. 
Let 
\[
 a_\sF^{nc}= \dfrac{ (\gamma_{ss,\sF}^- - \gamma_{se,\sF}^-) \gamma_{ee,\sF}  -(\gamma_{se,\sF}^- -\gamma_{ee,\sF}^-) \gamma_{se,\sF}  }
{\gamma_{ss,\sF}\gamma_{ee,\sF}- \gamma_{se,\sF}^2} >0 , \quad \mbox{and}
\]
\[
b_\sF^{nc}=\dfrac{  (\gamma_{se,\sF}^- -\gamma_{ee,\sF}^-) \gamma_{ss,\sF}
-(\gamma_{ss,\sF}^- - \gamma_{se,\sF}^-) \gamma_{se,\sF} }
{\gamma_{ss,\sF}\gamma_{ee,\sF}- \gamma_{se,\sF}^2} <0
\]
with $ \gamma_{ij,\sF}^\pm$ and $ \gamma_{ij,\sF}$, ($i,j \in \{s,e\}$) defined in Section 4.2.
Similar to the gradient recovery using the $\ND$ element for the mixed method, 
the approximation to the error gradient is 
\beq\label{app-err-g-nd-nc}
\rd_{nc, nd} 
= \sum_{F\in\cE_\sI} \left(\hat{a}^{nc}_{\sF} j^{nc}_{g,\sF}  \bpsi_{s,\sF}^{nd}
 + \hat{b}_\sF^{nc} j^{nc}_{g,\sF} \bpsi_{e,\sF}^{nd}\right)
-\sum_{F\in\cE_\sD} j^{nc}_{g,\sF} (\bpsi^{nd,-}_{s,\sF}-\bpsi^{nd,-}_{e,\sF}).
\eeq
where the coefficients $\hat{a}^{nc}_{\sF}$ and $\hat{b}^{nc}_{\sF}$  are given by
$$
  \hat{a}^{nc}_{\sF} = \left\{ \begin{array}{llll}
 1-a_\sF^{nc}, & \mbox{on}&K_{\sF}^-, \\[2mm]
 -a_\sF^{nc}, & \mbox{on}& K_{\sF}^+,
\end{array}
\right.
\quad\mbox{and}\quad
  \hat{b}^{nc}_{\sF} = \left\{ \begin{array}{llll}
 1+b_\sF^{nc}, & \mbox{on}&K_{\sF}^-, \\[2mm]
 b_\sF^{nc}, & \mbox{on}& K_{\sF}^+,
\end{array}
\right.
$$
Now, the recovered gradient using the $\ND$ element is given by
\beq
\brho_{nc}^{nd} = \rd_{nc, nd} + \hat{\brho}_{nc}^{nd} 
= \sum_{F\in\cE}\rho_{nc,s,\sF}^{nd} \bpsi_{s,\sF}^{nd}
 +\sum_{F\in\cE}\rho_{nc,s,\sF}^{nd} j^{nc}_{g,\sF} \bpsi_{e,\sF}^{nd},
\eeq
where the coefficients of $\bpsi_{s,\sF}^{nd}$ and $\bpsi_{e,\sF}^{nd}$ are given by
\beq
\hat{\rho}^{nd}_{nc, s,\sF} = \left\{
\begin{array}{lllll}
 a_\sF^{nc} \hat{\rho}^{-}_{nc,\sF}+ \left(1-a_\sF^{nc}\right)
 \hat{\rho}^{+}_{nc,\sF}, & F\in \cE_\sI,\\[2mm]
\grad g_\sD \!\!\cdot\bt_\sF,& F\in \cE_\sD,\\[2mm]
j^{nc}_{g,\sF}, & F\in \cE_\sN,
\end{array}
\right.
\eeq
and
\beq
\hat{\rho}^{nd}_{nc, e,\sF} = \left\{
\begin{array}{lllll}
 b_\sF^{nc} \hat{\rho}^{-}_{nc,\sF}+ \left(1+b_\sF^{nc}\right)
 \hat{\rho}^{+}_{nc,\sF}, & F\in \cE_\sI,\\[2mm]
-\grad g_\sD \!\!\cdot\bt_\sF,& F\in \cE_\sD,\\[2mm]
-j^{nc}_{g,\sF}, & F\in \cE_\sN.
\end{array}
\right.
\eeq

\section{Explicit A Posteriori Error Estimators}\label{estimators-a}
\setcounter{equation}{0}

With the explicit recoveries of the flux and gradient introduced in Section~4 for various
finite element approximations, this section describes the corresponding recovery-based 
a posteriori error estimators.

For the conforming linear element, we study two estimators using the respective 
$RT$ and $BDM$ recoveries. The global $RT$ a posteriori error estimator is given by
 \[
 \eta^{rt}_c = \|A^{-1/2}\left(\bsigma_c^{rt} + A \grad u_c\right)\|_{0,\O} = \|A^{-1/2} \sd_{c,rt} \|_{0,\O},
 \]
and the $RT$ local 
error indicators on element $K\in\cT$ and on edge $F\in\cE$ are given by
 \[
 \eta^{rt}_{c,\sK} 
 = \|A^{-1/2} \sd_{c,rt} \|_{0,K}  
 \quad\mbox{ and }\quad 
 \eta^{rt}_{c,\sF} =  \|A^{-1/2} \sd_{c,rt,\sF} \|_{0,\o_\sF},
 \]
respectively, where $\sd_{c,rt}$ is given in (\ref{app-err-f-rt})  and $\sd_{c,rt,\sF}$  
in (\ref{rt-N}) and (\ref{rt-I}). The global $\BDM$ a posteriori error estimator is given by
 \[
\eta^{bdm}_c = \|A^{-1/2} \left(\bsigma_c^{bdm} + A\grad u_c\right)\|_{0,\O}
= \|A^{-1/2}\sd_{c,bdm} \|_{0,\O},
\]
and the $BDM$ local 
error indicators on element $K\in\cT$ and on edge $F\in\cE$ are given by
 \[
 \eta^{bdm}_{c,\sK} 
 = \|A^{-1/2} \sd_{c,bdm} \|_{0,K}
 \quad\mbox{ and }\quad
 \eta^{bdm}_{c,\sF} =  \|A^{-1/2} \sd_{c,bdm,\sF} \|_{0,\o_\sF},
 \]
 respectively, where $\sd_{c,bdm}$ is defined in (\ref{app-err-f-bdm})  and $\sd_{c,bdm,\sF}$ in 
 (\ref{rt-N}) and (\ref{bdm-I}). 

For the lowest-order mixed element, we study one estimator based on the explicit $\ND$ recovery.
The local 
error indicators on element $K\in\cT$ and on edge $F\in\cE$ are defined by
 \[
 \eta^{nd}_{m,\sK} 
 = \|A^{1/2} \rd_{m,nd} \|_{0,K}
 \quad\mbox{ and }\quad
\eta^{nd}_{m,\sF} =  \|A^{1/2} \rd_{m,nd,\sF} \|_{0,\o_\sF},
 \]
 respectively, where $\rd_{m,nd}$ is defined in (\ref{app-err-g}) and 
 $\rd_{m,nd,\sF}$ in (\ref{app-err-g-F}) and (\ref{app-err-g-I}).
 The global error estimator is then defined by
 \[
 \eta^{nd}_m = \|A^{-1/2}\left(\bsigma_m + A \brho_m^{nd}\right)\|_{0,\O} = \|A^{1/2} \rd_{m,nd} \|_{0,\O}.
 \]

For the nonconforming linear element, again we introduce two estimators based on the $RT$-$\NE$
and $\BDM$-$\ND$ recoveries. Let  $c_1,\,\,c_2\in (0,1)$ be parameters to be determined such that $c_1+c_2=1$ 
(e.g, $c_1=c_2 =1/2$). The global
$RT$-$\NE$ error estimator is defined by
 \[
 \eta^{rh}_{nc} = \left(c_1 \|A^{-1/2} \sd_{nc,rt} \|_{0,\O}^2+c_2 \|A^{1/2} \rd_{nc,ne} \|_{0,\O}^2\right)^{1/2},
 \]
 and the $RT$-$\WH$ local error indicators on element $K\in\cT$ and on edge $F\in\cE$ are defined 
  respectively by
 \begin{eqnarray*}
 \eta^{rh}_{nc,\sK}
 &=& \left(\!c_1 \|A^{-1/2} \sd_{nc,rt} \|_{0,K}^2\!
 +c_2 \|A^{1/2} \rd_{nc,ne} \|_{0,K}^2\!\right)^{1/2}\\[2mm]
\mbox{and} \quad
\eta^{rh}_{nc, \sF}
&=& \left(c_1 \|A^{-1/2} \!\sd_{nc,rt,\sF} \|_{0,\o_\sF}^2
+c_2  \|A^{1/2} \rd_{nc,ne,\sF} \|_{0,\o_\sF}^2\right)^{1/2},
 \end{eqnarray*}
where $\sd_{nc,rt}$, $\sd_{nc,rt, \sF}$, $\rd_{nc,ne}$ and $\rd_{nc,ne,\sF}$ are defined in 
(\ref{sdncrtf}), (\ref{app-err-f-nc-F}), (\ref{app-err-g-nc}), and (\ref{local-rho-nc}), respectively.

Similarly, The global $\BDM$-$\ND$ error estimator is defined by
 \[
 \eta^{rh}_{nc} = \left(c_1 \|A^{-1/2} \sd_{nc,bdm} \|_{0,\O}^2+c_2\|A^{1/2} \rd_{nc,nd} \|_{0,\O}^2\right)^{1/2},
 \]
 and the local $\BDM$-$\ND$ error indicators on element $K\in\cT$ and on edge $F\in\cE$ are defined 
  respectively by
 \begin{eqnarray*}
 \eta^{rh}_{nc,\sK}
 &=& \left(\!c_1 \|A^{-1/2} \sd_{nc,bdm} \|_{0,K}^2\!
 +c_2 \|A^{1/2} \rd_{nc,nd} \|_{0,K}^2\!\right)^{1/2}\\[2mm]
\mbox{and} \quad
\eta^{rh}_{nc, \sF}
&=& \left(c_1 \|A^{-1/2} \sd_{nc,bdm,\sF} \|_{0,\o_\sF}^2
+c_2  \|A^{1/2} \rd_{nc,nd,\sF} \|_{0,\o_\sF}^2\right)^{1/2},
 \end{eqnarray*}
where $\sd_{nc,bdm}$ and $\rd_{nc,nd}$ are defined in 
 (\ref{bdm-Ib}) and (\ref{app-err-g-nd-nc}), respectively.

\section{Efficiency and Reliability}\label{estimators-b}
\setcounter{equation}{0}

This section establishes efficiency and reliability bounds of the estimators 
defined in Section 5 for interface problems (i.e, $A = \alpha\, I$
and $\alpha(x)$ is a piecewise constant with respect to the triangulation $\cT$.).
In order to show that the reliability constant are independent of the jump of $\alpha$,
as usual, we assume that the distribution of the coefficients $\a_\sK$
for all $K\in \cT$ is locally quasi-monotone \cite{Pet:02}, which is
slightly weaker than Hypothesis 2.7 in \cite{BeVe:00}.
For convenience of readers, we restate it here.

Let $\o_z$ be the union of all elements having $z$ as a vertex.
For any $z\in\cN$, let
 \[
 \hat{\omega}_z=\{K\in\omega_z \,:\, \a_\sK = \max_{\sK'\in\omega_z} \a_{\sK'}\}.
\]
\begin{defn}\label{defnquasimonotone}
Given a vertex $z \in \cN$, the distribution of the coefficients $\a_\sK$, $K\in\omega_z$,
is said to be {\em quasi-monotone} with respect to the vertex $z$
if there exists a subset $\tilde{\o}_{\sK,z,qm}$ of $\omega_z$ such that the union
of elements in $\tilde{\o}_{\sK,z,qm}$ is a Lipschitz domain and that
\begin{itemize}
\item if $z\in\cN\backslash\cN_\sD$, then $\{K\}\cup \hat{\o}_z
\subset \tilde{\o}_{\sK,z,qm}$
and $\a_\sK\leq \a_{\sK'} \; \forall \sK' \in \tilde{\o}_{\sK,z,qm}$;
\item if $z\in\cN_\sD$, then $K\in \tilde{\o}_{\sK,z,qm}$,
$\p\tilde{\omega}_{\sK,z,qm}\cap\Gamma_D \neq \emptyset$, and
$\a_\sK\leq \a_{\sK'} \; \forall \sK' \in \tilde{\o}_{\sK,z,qm}$.
\end{itemize}
The distribution of the coefficients $\a_\sK$, $K\in\cT$, is said to be
locally {\em quasi-monotone} if it is quasi-monotone with respect to
every vertex $z\in\cN$.
\end{defn}

Let $f_{_\cT}$ be the $L^2$ projection of $f$ onto the space of piecewise constant defined on elements of $\cT$, let
 \begin{eqnarray*}
 H_f &=& \left(\sum_{K\in\cT} H_{f,K}^2\right)^{1/2}
 \quad\mbox{with}\quad
 H_{f,K} = \dfrac{h_K}{\sqrt{\alpha_K}}\, \|f-f_{_\cT}\|_{0,K}
 \quad\forall\;K\in\cT,
\end{eqnarray*}
and let
 \[
\hat{H}_f=\left(\sum_{z\in \cN\cap (\calS\cup\Gamma_D)}
 \sum_{K\subset\omega_z} \dfrac{h^2_K}{\a_K}\,\|f\|^2_{0,K}
 +\sum_{z\in \cN\setminus (\calS\cup\Gamma_D)}
 \sum_{K\subset\omega_z} \dfrac{h^2_K}{\a_K}\,
 \|f-\dashint_{\omega_z}f\,dx\|^2_{0,K}\right)^{1/2},
 \]
where $\dashint_{\omega_z}f\,dx=\int_{\hat{\o}_z} f \psi_z\,dx\big/
\int_{\hat{\o}_z} \psi_z\,dx$ is a weighted average of $f$ over $\hat{\o}_z$
and $\psi_z$ is a linear nodal basis function at $z\in\cN$.


\begin{rem}  For various lower order finite element approximations,
the second term in $\hat{H}_f$ is of higher order
than $\eta_{_\cE}$ {\em (}defined below in {\em (\ref{edgeestimator-c}))}
  for $f\in L^2(\O)$ and so is the first term for $f\in L^p(\O)$
with $p > 2$ {\em (}see {\em \cite{CaVe:99}}{\em )}.
\end{rem}

\subsection{Conforming Elements}

\begin{thm}
Assume that the distribution of the coefficients are quasi-monotone. Then
the error estimators $\eta^{rt}_c$ and $\eta^{bdm}_c$ satisfy the 
global reliability bound:
\beq\label{rel-rt}
\|\alpha^{1/2} \grad e_c\|_{0,\O} \leq C(\eta^{rt}_c + \hat{H}_f )
\quad\mbox{and}\quad
\|\alpha^{1/2} \grad e_c\|_{0,\O} \leq C(\eta^{bdm}_c + \hat{H}_f),
\eeq
where the constants above are independent of $\alpha$ and the mesh size.
\end{thm}

\begin{proof} Inequalities (\ref{rel-rt}) may be established in a similar fashion as those 
in \cite{CaZh:09, CaZh:10c}.
\end{proof}

To prove the efficiency bound, consider the edge error estimator and indicator of the residual type:
\beq\label{edgeestimator-c}
\eta_{c,\cE} := \left(  \sum_{F\in\cE_\sI \cup\cE_\sN} \eta_{c,\sF}^2 \right)^{1/2}
\quad\mbox{with}\quad \eta_{c,\sF} = \left\{\begin{array}{lll}
h_\sF j^c_{f,\sF}\big/ \sqrt{\alpha_\sF^+ + \alpha_\sF^-}, & F\in\cE_\sI,\\[4mm]
  h_\sF j^c_{f,\sF}\big/ \sqrt{ \alpha_\sF^-},  & F\in \cE_\sN.
\end{array}
\right.
\eeq
Without assumptions on the distribution of the coefficient $\alpha$, it was proved
by Petzoldt (see equation (5.7) in \cite{Pet:02}) that there exists a constant $C>0$ independent of
$\alpha$ and the mesh size such that
\beq \label{edgeestimator}
\eta_{c,\sF}^2 \leq C\left( \|\alpha^{-1/2} \grad e_c\|_{\o_\sF}^2 + \sum_{K\in\cT_\sF} H_{f,\sK}^2
\right).
\eeq
Let $\cT_\sK = \{T\in\cT: T \mbox{ and } K \mbox{ share at least one edge} \}$.

\begin{thm}\label{thm:eff-c}
The local indicators $\eta^{rt}_{c,F}$, $\eta^{rt}_{c,K}$, $\eta^{bdm}_{c,F}$,
and $\eta^{bdm}_{c,K}$ defined in {\em Section 5} are efficient, 
i.e., there exists a constant $C>0$ independent of
$\alpha$ and the mesh size such that
\begin{eqnarray} \label{c-eff-F}
\eta^{bdm}_{c,\sF} \leq \eta^{rt}_{c,\sF}  &\leq& C \|\alpha^{1/2} \grad e_c\|_{0,\o_\sF}
+C\left( \sum_{K\in\cT_\sF} H_{f,\sK}^2\right)^{1/2}\\[2mm]
\label{c-eff-K1}
\mbox{and}\quad
\eta^{bdm}_{c,\sK}, \,\, \eta^{rt}_{c,\sK} & \leq& C \|\alpha^{1/2} \grad e_c\|_{0,\o_\sK}+
C\left( \sum_{T\in\cT_\sK} H_{f,\sT}^2\right)^{1/2} .
 \end{eqnarray}
\end{thm}

\begin{proof}
Without loss of generality, we  establish the efficiency bounds only for interior edges.
The first inequality of (\ref{c-eff-F}) is a direct consequence of 
the minimization problem in (\ref{c-mini}) and the fact that $RT^c_{-1,F}\subset BDM^c_{-1,F}$. 
To prove the second inequality of (\ref{c-eff-F}), we assume that the triangulation is regular. 
By the equivalence in (\ref{c-equi}) and the fact that $\|\bpsi^{rt}_\sF\|_{0,\o_\sF}\leq C\,h^2_\sF$, we have
 \begin{eqnarray*}
 \left(\eta^{rt}_{c,\sF} \right)^2
 &=& \|\alpha_{\sF^-}^{-1/2}\sd_{c,rt,\sF}\|_{0,K^-_\sF}^2
 + \|\alpha_{\sF^+}^{-1/2}\sd_{c,rt,\sF}\|_{0,K^+_\sF}^2  \\[2mm]
 &\leq& C\left( \dfrac{1}{\alpha_\sF^-} \left(\dfrac{\alpha_\sF^-}{\alpha_\sF^- + \alpha_\sF^+} \right)^2
 +  \dfrac{1}{\alpha_\sF^+} \left(\dfrac{\alpha_\sF^+}{\alpha_\sF^- + \alpha_\sF^+} \right)^2 \right)
 \left(j_{f,\sF}^c h_\sF\right)^2 = C\, \eta_{c,\sF}^2,
\end{eqnarray*}
which, combining with (\ref{edgeestimator}), implies the second inequality of (\ref{c-eff-F}). It is easy to see that
 \[
\left( \eta^{rt}_{c,\sK} \right)^2 \leq \sum_{F\in\cE_\sK} \left( \eta^{rt}_{c,\sF} \right)^2
\quad\mbox{and}\quad
\left( \eta^{bdm}_{c,\sK} \right)^2 \leq \sum_{F\in\cE_\sK} \left( \eta^{bdm}_{c,\sF} \right)^2.
\]
Now, (\ref{c-eff-K1}) follows from (\ref{c-eff-F}).
This completes the proof of the theorem.
\end{proof}

\subsection{Mixed Elements}

\begin{thm}
Assume that the distribution of the coefficient $\alpha$ is  quasi-monotone. Then
the error estimator $\eta^{nd}_{m}$ satisfies the following
global reliability bound:
\beq
\|\alpha^{-1/2}\bE_m\|_{0,\O} \leq C (\eta_m^{nd}+H_f + G_{\grad_h\times(\alpha^{-1}\bsigma_m)}),
\eeq
where $G_{\grad_h\times(\alpha^{-1}\bsigma_m)}$ is a higher order term if
$\grad_h\times(\alpha^{-1}\bsigma_m) \in L^p(\O)$ with $p>2$.
Moreover, if $\cV = RT$, then
\beq
\|\alpha^{-1/2}\bE_m\|_{0,\O} \leq C (\eta_m^{nd}+H_f).
\eeq
\end{thm}

\begin{proof}
Let $\hat{\eta}_m$ be the implicit recovery-based estimator introduced in \cite{CaZh:10a}, i.e., 
 \[
 \hat{\eta}_{m} = \min_{\btau \in \ND} \|\alpha^{1/2}\btau + \alpha^{-1/2}\bsigma_m\|_{0,\O}.
 \]
It is obvious that $\hat{\eta}_{m} \leq \eta^{nd}_{m}$. Now, the theorem is a direct consequence of 
Theorem 6.2 of \cite{CaZh:10a}.
\end{proof}

The efficiency of the $\eta^{nd}_{m}$ may be established by a direct calculation similar to the proof of
Theorem \ref{thm:eff-c}. However, the calculation is quite complicated in this case. We will prove it
through the following Helmholtz decomposition (see, e.g., \cite{GiRa:86}) of the error flux $\bE_m$: 
there exist $\xi_m\in H_\sD^1(\O)$ and 
$\zeta_m \in H^1_N(\O)\equiv \{v\in H^1(\O)\big| \,v = 0\mbox{ on }\Gamma_N\}$ such that
\begin{eqnarray}\label{H-decom}
&& \bE_m 
= \alpha \grad \xi_m+\gperp \zeta_m\\[2mm]\nonumber
\mbox{and}
&& \|\a^{-1/2}\bE_m\|^2_{0,\O}
=\|\alpha^{1/2}\grad \xi_m\|^2_{0,\O} + \|\alpha^{-1/2} \gperp \zeta_m\|^2_{0,\O}.
\end{eqnarray}

\begin{thm}
The local indicators $\eta^{nd}_{m,\sF}$ and
$\eta^{nd}_{m,\sK}$ and the global error estimator $\eta^{nd}_m$
are efficient, i.e., there exists a constant $C>0$ independent of
$\alpha$ and the mesh size such that
\begin{eqnarray} \label{m-eff}
\eta^{nd}_{m,\sF}  &\leq& C \|\alpha^{-1/2} \gperp\zeta_m\|_{0,\o_\sF},
\quad
 \eta^{nd}_{m,\sK}  \leq C \|\alpha^{-1/2} \gperp\zeta_m\|_{0,\o_\sK},\\[2mm]
\label{m-eff-G}
 \mbox{and}\quad   \eta^{nd}_{m}  &\leq & C \|\alpha^{-1/2} \gperp\zeta_m\|_{0,\O} \leq
C \|\alpha^{1/2} \bE_m\|_{0,\O}.
\end{eqnarray}
\end{thm}

\begin{proof}
Without loss of generality, we  establish the efficiency bounds only for interior edges.
Let $\eta_{m,\sF}$ and $\eta_{m}$ be the respective edge indicator and estimator defined 
in \cite{CaZh:10a}, where 
  \begin{equation}\label{6.11a}
 \eta_{m,\sF}^2=\dfrac{\a^-_\sF+\a^+_\sF}{2} h_\sF \int_\sF \large| j^m_{g,\sF}\large|^2\,ds.
 \end{equation}
 It is proved in Proposition 6.6 of \cite{CaZh:10a} that
\beq\label{mixedge}
\eta_{m,\sF} \leq \|\alpha^{-1/2} \gperp \zeta_m\|_{0,\omega_\sF} \;\mbox{   and   }\;
\eta_{m}\leq C\|\alpha^{-1/2} \gperp \zeta_m\|_{0,\Omega} \leq C\|\alpha^{1/2}\bE_m\|_{0,\Omega}.
\eeq
Since $\|\bpsi^{nd}_{i,F}\|_K\approx C\,h_\sF$ for $i=s,e$,
it follows from (\ref{mix-mini}) with $\btau=\bzero$, (\ref{rhojI}), and the triangle inequality that
\begin{eqnarray*}
\eta^{nd}_{m,\sF} & =& \|\alpha^{1/2} \rd_{m,nd,\sF}\|_{0,\omega_\sF}
\leq \|\alpha^{1/2} \brho^m_{j,\sF}\|_{0,\omega_\sF}
\\[2mm]
&=& \sqrt{\alpha^-_\sF} \left(|c_{s,\sF}|\, \|\bpsi^{nd,-}_{s,\sF}\|_{0,K_\sF^-}
+|c_{e,\sF}|\, \|\bpsi^{nd,-}_{e,F}\|_{0,K_\sF^-}\right)\\[2mm]
&\leq &C\, h_\sF\sqrt{\alpha^-_\sF}\left(|c_{s,\sF}|+|c_{e,\sF}|\right) .
\end{eqnarray*}
Note that 
 \[
 c_{s,\sF}= j^m_{g,\sF}(\bs_\sF)
 \quad\mbox{and}\quad
 c_{e,\sF}= j^m_{g,\sF}(\be_\sF)
 \]
and that $j^m_{g,\sF}$ is an affine function on $F$, it is then easy to check that
there exists 
 a constant $C>0$ independent of $\a$ and $h_\sF$ such that 
 \[
  |c_{s,\sF}|+|c_{e,\sF}|
    \leq C\,h_\sF^{-1/2} \left(\int_\sF \large| j^m_{g,\sF}\large|^2\,ds\right)^{1/2}.
 \]
By using the above two inequalities, we have
 \[
 \eta^{nd}_{m,\sF}
 \leq C\, h_\sF^{1/2}\sqrt{\alpha^-_\sF}  \left(\int_\sF \large| j^m_{g,\sF}\large|^2\,ds\right)^{1/2}
 \leq C\,  \eta_{m,\sF},
 \]
which, together with (\ref{mixedge}), implies the validity of the first inequality in (\ref{m-eff}).
Now, the second inequality in (\ref{m-eff}) and (\ref{m-eff-G}) are straightforward from
the definitions and (\ref{mixedge}).
\end{proof}

\subsection{Nonconforming Elements}

\begin{thm}
Assume that the distribution of the coefficient $\a$ is quasi-monotone. Then
the error estimators $\eta^{rh}_{nc}$ and $\eta^{bd}_{nc}$ satisfy the 
global reliability bounds:
 \begin{eqnarray}\label{nc-rel-1}
 && \|\alpha^{1/2} \grad_h e_{nc}\|_{0,\O} 
 \leq  C\left(\eta^{bd}_{nc} +H_f\right)\\[2mm] \label{nc-rel-2}
 \mbox{and} &&
\|\alpha^{1/2} \grad_h e_{nc}\|_{0,\O} \leq  C\left(\eta^{rh}_{nc} +H_f\right).
 \end{eqnarray}
\end{thm}

\begin{proof}
Let $\hat{\eta}_{nc}$ be the implicit recovery-based estimator introduced in \cite{CaZh:10a}: 
 \[
 \hat{\eta}_{nc}^2= c\, \hat{\eta}_{nc,1}^2+ (1-c)\hat{\eta}_{nc,2}^2
 \]
with $c\in (0,1)$ being a parameter to be determined, where 
 \[
 \hat{\eta}_{nc,1}=\min_{\btau \in \BDM} \|\alpha^{-1/2}\btau + \alpha^{1/2}\grad_h u_{nc}\|_{0,\O}
 \quad\mbox{and}\quad
 \hat{\eta}_{nc,2}=\min_{\btau \in \ND} \|\alpha^{1/2}(\btau - \grad_h u_{nc})\|_{0,\O}.
 \]
It is obvious that $\hat{\eta}_{nc,1}\leq  \eta^{bdm}_{nc}  \leq \eta^{rt}_{nc}$ and that
$\hat{\eta}_{nc,2}\leq  \eta^{nd}_{nc}  \leq \eta^{ne}_{nc}$. Now, (\ref{nc-rel-1})  and 
(\ref{nc-rel-2}) follow from Theorem 6.4 of \cite{CaZh:10a}.
\end{proof}

To prove the efficiency of the explicit error estimators, consider the weighted
edge error estimator introduced in \cite{CaZh:10a}:
\[
\eta_{nc,\cE} := \left( \sum_{F\in\cE} \eta_{nc,\sF}^2 \right)^{1/2}
\mbox{with}\quad
\eta_{nc,\sF}^2 = \left\{ \begin{array}{lll}
 \dfrac{2h_\sF^2}{\alpha_K^+ + \alpha_K^-} \left(j^{nc}_{f,\sF}\right)^2  +
 \dfrac{h_\sF^2\alpha_K^+ \alpha_K^-}{\alpha_K^+ + \alpha_K^-} \left(j^{nc}_{g,\sF}\right)^2, & F\in\cE_\sI, \\[4mm]
 \dfrac{h_\sF^2}{ \alpha_K^-} (j^{nc}_{f,\sF})^2, & F\in\cE_\sN, \\[4mm]
 h_\sF^2 \alpha_K^- \left(j^{nc}_{g,\sF}\right)^2, & F\in\cE_\sD.
\end{array}
\right.
 \]
 
 \begin{lem}\label{lem-nc1}
There exist a positive constant $C$ independent of $\a$ and the mesh size such that
\beq\label{edge-nc}
\eta^{rt}_{nc,\sF}\leq C\eta_{nc,\sF}
\quad \mbox{and}\quad
\eta^{ne}_{nc,\sF}\leq C\eta_{nc,\sF}
\eeq
\end{lem}

\begin{proof}
Without loss of generality, we  prove the validity of the lemma only for interior edges.
Assume that the triangulation is regular, then $\|\bphi^{rt}_\sF\|_{0,K}\leq C\,h_\sF$.
It follows from the definition of $\eta^{rt}_{nc,\sF}$, (\ref{sdncrtf}), and the equivalence (\ref{c-equi}) that
\begin{eqnarray*}
\eta^{rt}_{nc,1,\sF}
&=& \|\alpha^{-1/2} \sd_{nc,rt,\sF}\|_{0,\o_\sF}
 =\left( \|\alpha_{F^-}^{-1/2}\sd_{nc,rt,\sF}\|_{0,K^-_\sF}^2
 + \|\alpha_{F^+}^{-1/2}\sd_{nc,rt,\sF}\|_{0,K^+_\sF}^2\right)^{1/2}  \\[2mm]
 &\leq& C\,  h_\sF j_{f,\sF}^{nc}
 \left( \dfrac{1}{\alpha_\sF^-} \left(\dfrac{\alpha_\sF^-}{\alpha_\sF^- + \alpha_\sF^+} \right)^2
 +  \dfrac{1}{\alpha_\sF^+} \left(\dfrac{\alpha_\sF^+}{\alpha_\sF^- + \alpha_\sF^+} \right)^2 \right)^{1/2} \\[2mm]
 &=& C\,  \dfrac{h_\sF j_{f,\sF}^{nc}}{\sqrt{\alpha_\sF^- + \alpha_\sF^+}}
 \leq C \eta_{nc,\sF},
\end{eqnarray*}
which implies the first inequality in (\ref{edge-nc}). 

To prove the second inequality in (\ref{edge-nc}), for any $F\in\cE_\sI$, introduce
 \[
 \brho^{nc}_{j,\sF}=\left\{\begin{array}{ll}
                                        j^{nc}_{g,\sF} h_\sF\bpsi^{ne,-}_\sF, & \mbox{on }\, K^-_\sF,\\[3mm]
                                        0, & \mbox{on }\, K^+_\sF.
                                        \end{array}\right.
 \]
Without loss of generality, we assume that $\alpha_\sF^- \leq \alpha_\sF^+ $. (Otherwise, 
$\brho^{nc}_{j,\sF}$ may be redefined by exchanging $K^-_\sF$ and $K^+_\sF$.)
Since  $j^{nc}_{g,\sF}$ is a constant on $F$ and  $\|\bpsi^{ne}_\sF\|_K \approx C\,h_\sF$,
by the definitions of $\rd_{nc,ne,\sF}$ in (\ref{local-rho-nc}), we have
 \begin{eqnarray*}
 \eta^{ne}_{nc,\sF} 
 &=& \|\alpha^{1/2}\rd_{nc,ne,\sF}\|_{0,\o_\sF}
 \leq \|\alpha^{1/2}\brho^{nc}_{j,\sF}\|_{0,\o_\sF}
 = \sqrt{\alpha_\sF^-}\, \|\brho^{nc}_{j,\sF}\|_{0,K_\sF^-}\\[2mm]
 &= & \sqrt{\alpha_\sF^-}\,  j^{nc}_{g,\sF}  \|\bpsi^{ne}_\sF\|_{0,K_\sF^-} 
 \leq C \sqrt{\alpha_\sF^-}\,  j^{nc}_{g,\sF} h_\sF \\[2mm]
& \leq & C  \left(\dfrac{\alpha_K^+ \alpha_K^-}{\alpha_K^+ + \alpha_K^-} \right)^{1/2} j^{nc}_{g,\sF} h_\sF
 \leq  C \eta_{nc,\sF}.
\end{eqnarray*}
This completes the proof of the second inequality in (\ref{edge-nc}) and, hence, the lemma.
\end{proof}

\begin{thm}
The local indicators $\eta^{rh}_{nc,\sF}$, $\eta^{rh}_{nc,\sK}$, $\eta^{bd}_{nc,\sF}$, and
 $\eta^{bh}_{nc,\sK}$
 are efficient, i.e., there exists a constant $C>0$ independent of $\alpha$ and the mesh size such that
 \begin{equation} \label{nc-eff-F}
\eta^{bd}_{nc,\sF} \leq \eta^{rh}_{nc,\sF}  
\leq C\, \|\alpha^{1/2} \grad_h e_{nc}\|_{0,\o_\sF}
+C\left( \sum_{K\in\cT_\sF} H_{f,\sK}^2\right)^{1/2}
\end{equation}
and that 
\beq\label{nc-eff-K}
 \eta^{rh}_{nc,\sK},\,\, \eta^{bd}_{nc,\sK} 
 \leq C\, \|\alpha^{1/2} \grad_h e_{nc}\|_{0,\o_\sK}+
C\left( \sum_{T\in\cT_\sK} H_{f,\sT}^2\right)^{1/2}.
\eeq
\end{thm}

\begin{proof}
Let 
 \[
 \cV^{nc}_{-1,\sF}
 = \{ \btau\in\! L^2(\omega_\sF\!) \big|\, \btau|_\sK\!\!\in\!\cV(K)\,\,\forall\, K\in \cT_\sF,
\,  \jump{\btau\cdot\bn_\sF}_\sF\!\!
 =\! -j^{nc}_{f,\sF},\,
 \, \btau\cdot\bn_{_E}\! = 0\,\,
 \mbox{on}\, \,E \in\! \cE_{b, \sF}\!\}
 \]
with $\cV=RT\,\mbox{or}\, \BDM$ and let 
 \[
  \cW^{nc}_{-1,\sF}
  = \{ \btau\in\! L^2(\omega_\sF\!) \big|\, \btau|_\sK\!\!\in\!\cW(K)\,\,\forall\, K\in \cT_\sF,
\,  \jump{\btau\cdot\bt_\sF}_\sF\!\!
 =\! -j^{nc}_{g,\sF},\,
 \, \btau\cdot\bt_{_E}\! = 0\,\,
 \mbox{on} \,\,E \in\! \cE_{b, \sF}\!\}.
 \]
 with $\cW=\NE\,\mbox{or}\, \ND$.
Similar to Section 4.1, the approximation error fluxes $\sd_{nc,v,\sF}$ with $v= rt\mbox{ or } bdm$
and the approximation error gradients $\rd_{nc,w,\sF}$ with $w=ne\mbox{ or }nd$ 
are then the solutions of the minimization problems:
 \begin{eqnarray}\label{mini-flux-nc-1}
 && \|A^{-1/2} \sd_{nc, v,\sF} \|_{0,\o_\sF}
 =\min_{\btau\in\cV^{nc}_{-1,\sF}}\|A^{-1/2} \btau\|_{0,\o_\sF} \\[2mm] \label{mini-flux-nc-2}
 \,\mbox{ and } && 
 \|A^{1/2} \rd_{nc,w,\sF} \|_{0,\o_\sF}
 =\min_{\btau\in\cW^{nc}_{-1,\sF}}\|A^{1/2} \btau\|_{0,\o_\sF},
 \end{eqnarray}
respectively. Since 
 $RT^{nc}_{-1,F}\subset \BDM^{nc}_{-1,F}$ and 
$\NE^{nc}_{-1,F}\subset \ND^{nc}_{-1,F}$, 
the first inequality in (\ref{nc-eff-F}) follows from their definitions.
The second  inequality in (\ref{nc-eff-F}) is from the minimization problems in (\ref{mini-flux-nc-1})
and (\ref{mini-flux-nc-2}),
Lemma \ref{lem-nc1}, and Theorem 6.8 of \cite{CaZh:10a}. The bounds in (\ref{nc-eff-K})
 are straightforward from their definitions and inequality (\ref{nc-eff-F}).
\end{proof}

 \section{Numerical Experiments}
\setcounter{equation}{0}

In this section, we report some numerical results for an interface
problem with intersecting interfaces used by many authors, e.g.,
\cite{Kim:07,CaZh:09,CaZh:10a,CaZh:10b}, which is
considered as a benchmark test problem.
For simplicity, we only test the conforming element with explicit RT recovery. 
Other cases behave similarly. 

Let $\O=(-1,1)^2$ and
 \[
 u(r,\theta)=r^{\gamma}\mu(\theta)
 \]
in the polar coordinates at the origin with $\mu(\theta)$ being a
smooth function of $\theta$ \cite{CaZh:09}. The function
$u(r,\theta)$ satisfies the interface equation with $A= \a I$,
$\Gamma_N=\emptyset$, $f=0$, and
 \[
 \a(x)=\left\{\begin{array}{ll}
 R & \quad\mbox{in }\, (0,1)^2\cup (-1,0)^2,\\[2mm]
 1 & \quad\mbox{in }\,\O\setminus ([0,1]^2\cup [-1, 0]^2).
 \end{array}\right.
 \]
The $\gamma$ depends on the size of the jump.
In our test problem, $\gamma=0.1$ is chosen and is corresponding to
$R\approx 161.4476387975881$.
Note that the solution $u(r,\theta)$ is only in
$H^{1+\gamma-\epsilon}(\O)$ for any $\epsilon>0$ and, hence, it is
very singular for small $\gamma$ at the origin. This suggests that
refinement is centered around the origin.

\begin{figure}[!hts]
    \hfill
    \begin{minipage}[!hbp]{0.48\linewidth}
        \centering
        \includegraphics[width=0.99\textwidth,angle=0]{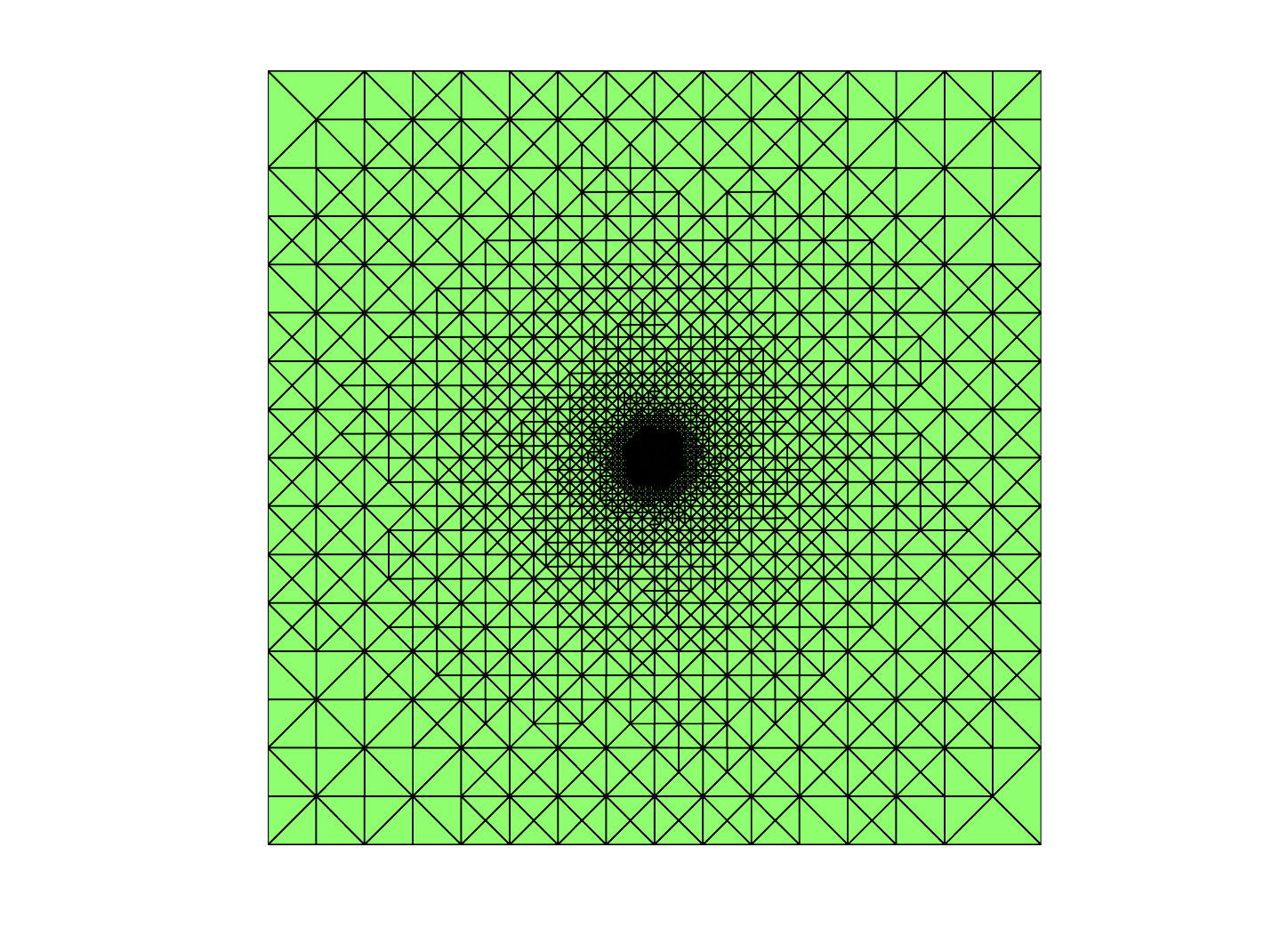}
        \caption{mesh generated by $\eta$}%
        \end{minipage}%
        \quad
    \begin{minipage}[!htbp]{0.48\linewidth}
        \centering
        \includegraphics[width=0.99\textwidth,angle=0]{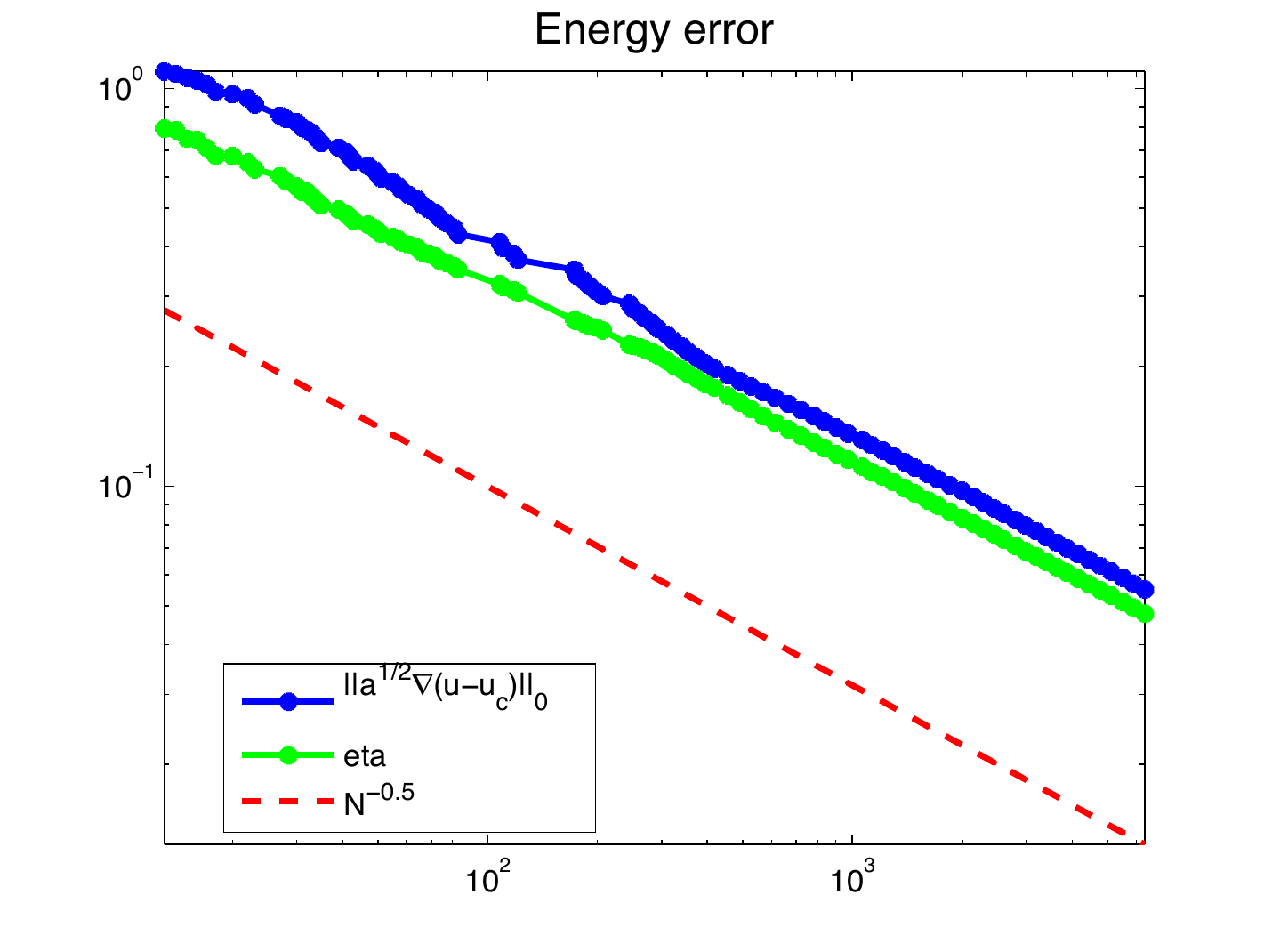}
        \caption{error and estimator $\eta$}%
    \end{minipage}%
        \hfill
\end{figure}

Mesh generated by $\eta_c^{rt}$ is shown in Figure 1. The refinement is centered at origin. Similar meshes for this test problem generated by other error estimators can be found in \cite{CaZh:09,CaZh:10a,CaZh:11}. 
The comparison of the error and the $\eta_c^{rt}$ is shown in Figure 2. The effectivity index is close to $1$. Moreover, the slope of the log(dof)- log(relative error) for $\eta_c^{rt}$ is $-1/2$, which indicates the optimal decay of the error with respect to the number of unknowns.

\appendix
\section{Basis Functions of the Lowest Order $RT$, $\BDM$, $\NE$, and $\ND$ Spaces}
\setcounter{equation}{0}

This appendix describes basis functions for the $RT$, $\BDM$, $\NE$,
and $\ND$ finite element spaces of the lowest order. The definition of these basis functions can also be found in Section 2.6 of \cite{BoBrFo:13}.

For a triangle $K$, denote by $\bx_i$, $\bx_j$, and $\bx_k$ its
three vertices sorted counterclockwise and denote by $F_i$, $F_j$,
and $F_k$ the edges opposite to the vertices $\bx_i$, $\bx_j$, and
$\bx_k$, respectively. 
The lengths, the unit tangent vectors, and the heights of the edges
are denoted by
 \[
 h_l=|\be_l|, \quad
 \bt_l=\dfrac{\be_l}{h_l}, \quad\mbox{and}\quad
 H_l
 \]
for $l=i,\,j,\,k$, respectively. Let $\lambda_i$, $\lambda_j$, and
$\lambda_k$ denote the barycentric coordinates of the triangle $K$
associated with vertices $\bx_i$, $\bx_j$, and $\bx_k$,
respectively. Then the unit outward vectors normal to the edges are
 \[
 \bn_l=-\dfrac{\grad \lambda_l}{ |\grad \lambda_l|}
 \quad\mbox{for }\,\, l=i,\,j,\,k.
 \]
Finally, denote by $|K|$ the area of the triangle $K$.
Now, we state basis functions associated with the edge $F_k$ as
follows:
 \begin{itemize}
 \item for RT
 \beq\label{rt}
 \bphi^{rt}_{_{F_k}}|_\sK  := \dfrac{1}{H_k}(\bx-\bx_k),
 \eeq
 \item for BDM (two basis functions associated with vertices $\bx_i$ and $\bx_j$)
 \beq\label{bdm}
 \bphi^{bdm}_{_{i,F_k}}|_\sK :=\  \dfrac{1}{H_k}(\bx_i-\bx_k)\lambda_i,
  \quad\mbox{and}\quad
\bphi^{bdm}_{_{j,F_k}}|_\sK :=\  \dfrac{1}{H_k}(\bx_j-\bx_k)\lambda_j,
 \eeq
 \item for $\NE$
 \beq\label{ne}
 \bphi^{ne}_{_{F_k}} |_\sK
 := h_k (\lambda_j \grad \lambda_i -\lambda_i \grad\lambda_j),
 \eeq
 \item for $\ND$ (two basis functions associated with vertices $\bx_i$ and $\bx_j$)
 \beq\label{nd}
  \bphi^{nd}_{_{i,F_k}}|_\sK
 :=  h_k \lambda_i \grad\lambda_j, 
 \quad\mbox{and}\quad
 \bphi^{nd}_{_{j,F_k}} |_\sK
 :=   h_k \lambda_j \grad\lambda_i. 
  \eeq
 \end{itemize}
It is easy to check that these basis functions satisfy the following
properties:
 \begin{itemize}
 \item for $RT$
 \beq \label{proprt}
 \bphi^{rt}_{_{F_k}} = \bphi^{bdm}_{_{i,F_k}} +\bphi^{bdm}_{_{j,F_k}} 
  \quad\mbox{and}\quad
 \left(\bphi^{rt}_{_{F_k}} \cdot
 \bn_{\ell}\right)|_{_{F_\ell}}
 = \delta_{\ell k}, \quad \ell = i,j,k;
 \eeq
 \item for $BDM$, 
  \beq 
 \left(\bphi^{bdm}_{_{i,F_k}} \cdot
 \bn_{\ell}\right)|_{_{F_\ell}}
 = \lambda_i \delta_{\ell k}
    \quad\mbox{and}\quad
 \left(\bphi^{bdm}_{_{j,F_k}} \cdot
 \bn_{k}\right)|_{_{F_k}} = \lambda_j \delta_{\ell k}, \quad \ell = i,j,k;
 \eeq
 Then it is clear that a linear function on $F_k$ can be represented by
 $\left(\bphi^{bdm}_{_{i,F_k}}\cdot
 \bn_{k}\right)|_{_{F_k}}$ and $\left(\bphi^{bdm}_{_{j,F_k}}\cdot
 \bn_{k}\right)|_{_{F_k}}$.
 Let $p$ be an affine function on $F_k$, then
 \beq \label{propbdm}
 p = p(\bx_i) \left(\bphi^{bdm}_{_{i,F_k}}\cdot\bn_k\right)|_{_{F_k}}
 +  p(\bx_j)
 \left(\bphi^{bdm}_{_{j,F_k}}\cdot\bn_k\right)|_{_{F_k}};
 \eeq
 \item for $\NE$
 \beq\label{propwh}
  \bphi^{ne}_{_{F_k}} =   \bphi^{nd}_{_{i, F_k}}-  \bphi^{nd}_{_{j, F_k}}
      \quad\mbox{and}\quad
 \left(\bphi^{ne}_{_{F_k}} \cdot \bt_{\ell}\right)|_{_{F_\ell}} =
 \delta_{\ell k}, \quad \ell = i, j, k;
 \eeq
 \item for $\ND$, 
   \beq  \label{propnd0}
 \left(\bphi^{nd}_{_{i,F_k}} \cdot
 \bt_{\ell}\right)|_{_{F_\ell}}
 = \lambda_i \delta_{\ell k}, 
       \quad\mbox{and}\quad
\left(\bphi^{nd}_{_{j,F_k}} \cdot
 \bt_{\ell}\right)|_{_{F_\ell}}
 =  -\lambda_j \delta_{\ell k}.
 \eeq
 Let $p$ be an affine function on $F_k$, then
 \beq \label{propnd}
 p = p(\bx_i) \left(\bphi^{nd}_{_{i,F_k}}\cdot\bt_k\right)|_{_{F_k}}
  - p(\bx_j)  \left(\bphi^{nd}_{_{j,F_k}}\cdot\bt_k\right)|_{_{F_k}}.
  \eeq
 \end{itemize}

\end{document}